\documentclass[draft]{article}

\usepackage{geometry}
\usepackage{amssymb,amsmath}

\title{A simple computational interpretation of set theory}
\author{Daniel M\'ehkeri\\dmehkeri@gmail.com}
\date{21 Feb 2011}

\newtheorem{thm}{Theorem}[section]
\newtheorem{mthm}[thm]{Meta-theorem}
\newtheorem{lem}[thm]{Lemma}
\newtheorem{cor}[thm]{Corollary}
\newtheorem{defn}[thm]{Definition}

\newenvironment{proof}{\textbf{Proof}:}{$\blacksquare$}
\newenvironment{rem}{\par\noindent\textbf{Remark}:}{}

\numberwithin{equation}{section}

\def\allin#1#2{\forall #1\!\in\!#2. \;}
\def\somein#1#2{\exists #1\!\in\!#2. \;}
\def\allset#1{\allin{#1}{\mathbb{V}}}
\def\someset#1{\somein{#1}{\mathbb{V}}}
\def\allnat#1{\allin{#1}{\mathbb{N}}}
\def\somenat#1{\somein{#1}{\mathbb{N}}}

\def\setb#1#2{\{#1\; | \; #2\}}
\def\setbn#1#2#3{\setb{#2}{#1\!\in\!\mathbb{N}\land#3}}
\def\empty{\varnothing}
\def\pow#1{\mathbb{P}(#1)}
\def\neumarral#1{\underline{#1}}
\def\p1{\pow{\neumarral{1}}}
\def\formalsys#1{$\mathbf{#1}$}

\def\tuple#1{\langle#1\rangle}
\def\stuple#1{\langle\!\!\circ#1\circ\!\!\rangle}
\def\defd#1{{#1\!\!\downarrow}}
\def\abs#1{\left|#1\right|}

\def\cd#1#2{#1 \diamond #2}
\def\el#1#2{#1\left[#2\right]}

\begin{document}
\maketitle

\begin{abstract}
\formalsys{CZF} is a system of set theory which, over classical logic, is equivalent to \formalsys{ZF}, while over intuitionistic logic, it has a well-known constructive type-theoretic interpretation. This article introduces a simpler, intuitive family of constructive interpretations: sets are {\it well-founded extensional computable conditional enumerations of sets}. One interpretation in this family is just this: all sets are inductively built from $\empty$ by iterating the construction $\setbn{n}{f_n}{g_n=h_n}$, where, in turn, $g$ and $h$ are computable sequences of sets, and $f$ is a computable sequence such that $f_n$ is a set when $g_n$ and $h_n$ are extensionally equal. Extended Church's Thesis, an assumption which is incompatible with classical logic, is required to make this a model of \formalsys{CZF}. Besides its foundational interest, it yields direct conservativity proofs for certain choice principles, the Subcountability axiom, and for some so-called Omniscience principles, including first-order arithmetic Omniscience. A larger interpretation in this family also models the Regular Extension Axiom. 

\end{abstract}
\section{Introduction}

Set theory studies more or less arbitrary collections that are well-founded and extensional. It provides a simple, unifying framework for all mathematics; for ordinary classical mathematics, the formal system of Zermelo-Fraenkel set theory with Choice (\formalsys{ZFC}), suffices. From the point of view of constructive and predicative mathematics, \formalsys{ZFC} is not suitable, but can be made so with remarkably little change. To this end, constructive Zermelo-Fraenkel set theory (\formalsys{CZF}) was introduced by Aczel. Over classical logic, \formalsys{CZF} has the same theorems as Zermelo-Fraenkel set theory without Choice (\formalsys{ZF}), and in fact its axiomatisation is not far from the way \formalsys{ZF} is usually axiomatised. 

To briefly describe the differences between \formalsys{ZF} to \formalsys{CZF}: most importantly of course \formalsys{CZF} uses intuitionistic logic, so the principle of the Excluded Middle (\formalsys{EM}) is absent. Extensionality, Pairing, Union, and Infinity are left as is. Replacement, Foundation, and Power Set in \formalsys{ZF} become Strong Collection, Bounded Separation, Set Induction, and Subset Collection in \formalsys{CZF} \cite[\S 2]{AR}

These substitutions are of no effect classically. Constructively, unacceptable principles are weakened. Foundation can be made acceptable by a contraposition: {\it no inhabited set intersects all of its elements}. Power Set can be made acceptable by switching to Exponentiation: {\it given two sets, there is a set of all functions from one set to the other}. This is weaker since constructively there is no ``set of all truth values'' to exponentiate. But these can then acceptably be strengthened. Constructive Foundation becomes Set Induction; the latter does not follow from the former constructively. Exponentiation becomes Subset Collection. Replacement can also be strengthened to Strong Collection and Bounded Separation. Separation does not follow from Replacement constructively. 

The infamous axiom of Choice is dropped entirely in \formalsys{CZF}, but the principle of Dependent Choices is widely used even in constructive and predicative mathematics, and is often taken as an additional axiom on top of \formalsys{CZF}. Two strengthenings of it are used as well. One is Relativised Dependent Choices (\formalsys{RDC}) \cite[\S 8.2]{AR}, extending over classes instead of just sets. Classically this is redundant, being just a consequence of \formalsys{DC}. Another is the Presentation principle \cite[\S 8.3]{AR}, which asserts that every set is the image of a so-called {\it base}. This is also known as the existence of Enough Projectives (\formalsys{EP}) in the category of sets. Not much is known about \formalsys{EP} classically. It is a consequence of full Choice, it implies \formalsys{DC}, and the latter implication cannot be reversed; Rathjen remarks, however, that it is an open problem whether \formalsys{EP} implies full Choice \cite[\S 5]{MR}.

It is not claimed that the axioms of \formalsys{CZF} are self-evidently constructively or predicatively acceptable. (It is not self-evident that they are classically acceptable, for that matter.) Rather, as shown by Aczel, they have an interpretation in Martin-L\"of's type-theoretic framework \cite{PA}, which also provides an interpretation for \formalsys{RDC} and \formalsys{EP} \cite{PA2}. Martin-L\"of's framework is a logic-free system of constructions that can be given a strong justification on its own \cite{PML}; alternately, a computational interpretation of sets can be obtained essentially by composition of the interpretation of sets in terms of types with a computational interpretation of types \cite[\S XII]{MB}.

This article introduces a simplified and more intuitive approach to the computational interpretation of sets, and also explicitly states the numerical principles on which the interpretation depends. The general definition of a set is just this:

\begin{defn}\label{def-g-set}Sets are well-founded extensional computable conditional enumerations of sets. Specifically:\begin{itemize}
\item If $p$ is a computable sequence of Meaningful conditions, and $f$ is a computable sequence such that $f_n$ is a set if $p_n$ is True, then $\setbn{n}{f_n}{(p_n\text{ is True})}$ is a set.
\item All sets are built inductively in this way.
\item Extensional equality between sets has the expected recursive definition.
\end{itemize}
\end{defn}
This definition is more specific than the classical concept, but by itself it is still too vague to be constructively useful, as it makes reference to Meaning and Truth. Fortunately, no insight into these great open problems of philosophy will be required. Instead, they should be treated as undefined words (and will be kept capitalised). The utility of this definition comes from the fact that usable set theory can be derived using relatively straightforward assumptions on Meaning. In particular, if the Meaningful conditions are taken to be extensional equations between two prior sets (plus a trivially false condition to get the induction off the ground), then a simple and mathematically precise definition is obtained:
\begin{defn}\label{def-e-set}$\;$\begin{itemize}
\item $\empty$ is a set. 
\item If $g$ and $h$ are computable sequences of sets, and if $f$ is a computable sequence such that $f_n$ is a set if $g_n = h_n$, then $\setbn{n}{f_n}{g_n=h_n}$ is a set.
\item All sets are built inductively from $\empty$ in this way.
\item Extensional equality between sets has the expected recursive definition.
\end{itemize}\end{defn}
To complete this definition needs only familiar and precise concepts: computability, inductive definition, as well as the meaning of extensional equality and the set builder notations. The inductive definition is no longer a strict accessibility definition, since extensional equality now interacts with the generating clauses for sets. Still, this has an intuitive justification: extensional equality is recursively defined in terms of prior sets. This is the same intuition that motivates the schema of simultaneous inductive-recursive definition \cite{PD}, and in fact this definition will be shown to fit into that schema.

A computable sequence of natural numbers can of course be coded by a natural number, so this can be expressed entirely in terms of natural numbers. It can be formalised in first-order arithmetic with intuitionistic logic (Heyting arithmetic, \formalsys{HA}) plus some axioms for the inductive-recursive definition. It leaves no ambiguity as to what a set is. And, with the assistance of Extended Church's Thesis (\formalsys{ECT}), all of \formalsys{CZF+RDC+EP} can be proven from it. \formalsys{ECT} is incompatible with classical logic, but it is well-known to be conservative over \formalsys{HA}, and will be shown to also be conservative over the inductive-recursive axioms.

Definition \ref{def-g-set}, on the other hand, allows progressively stronger assumptions on Meaning, and is easier to work with formally. This paper will therefore deal mainly with sets as general well-founded extensional computable conditional enumerations. Section 2 starts with no assumptions on Meaning whatsoever, and establishes the logical infrastructure of set theory. In section 3, with only trivial assumptions on Meaning, the basic sets are constructed. In section 4, \formalsys{ECT} is added to these trivial assumptions, and almost all of \formalsys{CZF+RDC+EP} is proven. A non-trivial assumption on Meaning in section 5 allows the proof to be completed.

In section 6 it is shown that the proof of \formalsys{EP} can in fact be strengthened to show the existence of Enough {\it Subcountable} Projectives (\formalsys{ESP}). This results in consequences that are incompatible with \formalsys{ZFC}, and some basic ones are listed.

The next three sections go backwards. Sections 7 reverses the interpretation back into \formalsys{CZF}, providing direct proofs of conservativity results for \formalsys{RDC} and \formalsys{ESP}. Section 8 gives a different reversal into \formalsys{ID_1}, a classical theory of inductive definitions. Section 9 takes up Definition \ref{def-e-set}, and shows that it is subsumed under Definition \ref{def-g-set} under exactly the assumptions of section 5.

Then, section 10 continues beyond \formalsys{CZF} with an even stronger assumption on Meaning to capture Aczel's {\it Regular Extension} axiom (\formalsys{REA}), and section 11 reverses the new interpretation back into \formalsys{CZF+REA}.

Finally section 12 discusses \formalsys{ECT}, and proves conservativity for some restricted forms of \formalsys{EM} called Omniscience principles, such as the decidability of all sentences of first-order arithmetic and of the well-foundedness of computable relations. These may be of interest from a weakly Platonist perspective which attributes objectively determinate truth or falsity to statements involving natural numbers, but not necessarily to those involving arbitrary sets of natural numbers, or involving sets in general.

\section{Logical principles}

From now on, for clarity, the term ``v-set'' will be used to denote the numerical code for a set according to Definition \ref{def-g-set}. In addition to the language of first-order arithmetic, there will be three unary predicates, $\mathcal{V}$, $\mathcal{M}$, and $\mathcal{T}$. $\mathcal{V}(x)$ is to be read ``$x$ is a v-set'', $\mathcal{M}(p)$ is to be read ``$p$ is a Meaningful condition'', and $\mathcal{T}(p)$ is to be read ``$p$ is a True condition''. With these, it can be specified exactly how a sentence in the language of first-order set theory is to be interpreted.

The Cantor pairing function from $\mathbb{N}^2$ to $\mathbb{N}$ will be denoted by the binary bracket $\tuple{\cdot,\cdot}$. $(\cdot)_L$ and $(\cdot)_R$ denote the corresponding left and right projections. The Kuratowski set-theoretic ordered pair will be denoted $\stuple{\cdot,\cdot}$ to distinguish it from the number-theoretic ordered pair. The coding of a natural number as a v-set (finite von Neumann ordinal) will be denoted $\neumarral{n}$, to distinguish it from $n$ as a natural number. 

Kleene's $\mathbf{T}$ predicate and $\mathbf{U}$ function will be used. $\mathbf{T}(e,i,o)$ asserts that the computation of the $e^{th}$ function on input $i$ terminates, and that $\mathbf{U}(o)$ is its output. Function application will be written by juxtaposition; the sentence $\defd{ei}$ means that $\somenat{o} \mathbf{T}(e,i,o)$, and the partial term $ei$ is equal to $\mathbf{U}(o)$ in that case (and will only be used in contexts where it is defined). The term $\Lambda n.\tau(n)$ is a number for the function computing the expression $\tau$.

The convention that a quantifier binds as far to the right as possible is adopted. A period will be placed after the quantifier as a reminder of this.

The predicate $\mathcal{V}$ is inductively defined following Definition \ref{def-g-set}. First introduce an abbreviation $\cd{x}{n}$ for the $n^{th}$ condition of the v-set $x$, and $\el{x}{n}$ for its $n^{th}$ conditional element, as follows:
\begin{align*}
\cd{x}{n} &= (xn)_R \\
\el{x}{n} &= (xn)_L
\end{align*}
(A slight liberty is taken with Definition \ref{def-g-set}, in that the pair of computable sequences has been switched for a computable sequence of pairs.) Assume: \begin{align}
\allnat{x} &(\allnat{n} \defd{xn} \wedge \mathcal{M}(\cd{x}{n}) \wedge (\mathcal{T}(\cd{x}{n}) \to \mathcal{V}(\el{x}{n}))) \to \mathcal{V}(x) \label{eq-vi}\\
(\allnat{x} &(\allnat{n} \defd{xn} \wedge \mathcal{M}(\cd{x}{n}) \wedge (\mathcal{T}(\cd{x}{n}) \to \phi(\el{x}{n}))) \to \phi(x)) \to (\allnat{x} \mathcal{V}(x) \to \phi(x)) \label{eq-ve}
\end{align}
\eqref{eq-ve} is a schema in $\phi$. The predicates $\mathcal{M}$ and $\mathcal{T}$ are for the moment completely unspecified, as discussed in the previous section. Otherwise, $\mathcal{V}$ is just given by an accessibility definition, which is immediately constructively valid. 

$IND$ will be an abbreviation for the indices for which the conditions of a v-set are True, and $EL$ will abbreviate the elements of a v-set, defined as follows:\begin{align*}
IND(x) &= \setb{n \in \mathbb{N}}{\mathcal{T}(\cd{x}{n})}\\
EL(x) &= \setb{\el{x}{n}}{n \in IND(x)}
\end{align*}
To be clear, $IND(x)$ and $EL(x)$ are subclasses of $\mathbb{N}$ rather than v-sets in the sense being defined here. In particular, membership in $EL(x)$ does not respect extensionality, rather, it ranges over the ``intensional members'' $\el{x}{n}$. Finally, for quantification over all v-sets, define $\mathbb{V}$ as another subclass of $\mathbb{N}$ as follows:
$$\mathbb{V} = \setb{n \in \mathbb{N}}{\mathcal{V}(n)}$$

For clarity, $\simeq$ will represent the extensional equality between v-sets. The regular $=$ will represent the decidable, intensional equality between natural numbers. For $\simeq$ it will be necessary to introduce a temporary fix-point definition. Assume:\begin{align}\begin{split}
\allnat{x,y} \mathcal{V}(x) \wedge \mathcal{V}(y) \to (x \simeq y \leftrightarrow &(\allin{a}{EL(x)} \somein{b}{EL(y)} a \simeq b) \\
\wedge & (\allin{b}{EL(y)} \somein{a}{EL(x)} b \simeq a))
\end{split}\label{eq-sim}\end{align}
This assumption is subsumed in section 5. For the membership relation, it will be more convenient not to take it as primitive. Rather, there will be two types of primitive quantifiers, bounded quantifiers over a v-set, interpreted as ranging over $EL(x)$, and unbounded quantifiers, interpreted as ranging over $\mathbb{V}$. $x \in y$ is defined as $\somein{x'}{y} x \simeq x'$. With these, the formal interpretation of set theory is complete. 

\begin{thm}[Extensionality]\label{thm-g-ext} If two v-sets have equal members, they are equal.\end{thm}
\begin{proof} Given the above interpretation, \eqref{eq-sim} is in fact: 
$$\allset{x,y} x \simeq y \leftrightarrow (\allin{a}{x} a \in y) \wedge (\allin{b}{y} b \in x)$$
\end{proof}

\begin{thm}[Set Induction]\label{thm-g-ind} If a predicate applies to a v-set whenever it applies to all its members, then the predicate applies to all v-sets.\end{thm}
\begin{proof} The inductive definition of $\mathcal{V}$ was chosen to make it so. For, if $\phi$ applies to a v-set whenever it applies to all its members, then:
$$\allset{x} (\allin{y}{x} \phi(y)) \to \phi(x)$$
which means,
$$\allnat{x} \mathcal{V}(x) \to (\allin{y}{EL(x)} \phi(y)) \to \phi(x)$$
Define the predicate $\phi'$ as:
$$\phi'(x) \iff \mathcal{V}(x) \to \phi(x)$$
Note that the elements of a v-set are themselves v-sets. So:
$$\allnat{x} (\allin{y}{EL(x)} \phi'(y)) \to \phi'(x)$$
which expands to
$$\allnat{x} (\allnat{n} \mathcal{T}(\cd{x}{n}) \to \phi'(\el{x}{n})) \to \phi'(x)$$
Furthermore $\mathcal{V}(x) \to (\allnat{n} \defd{xn} \wedge \mathcal{M}(\cd{x}{n}))$, so,
$$\allnat{x} (\allnat{n} \defd{xn} \wedge \mathcal{M}(\cd{x}{n}) \wedge (\mathcal{T}(\cd{x}{n}) \to \phi'(\el{x}{n}))) \to \phi'(x)$$
This is the form to which \eqref{eq-ve} applies, so,
$$\allnat{x} \mathcal{V}(x) \to \phi'(x)$$
So finally,
$$\allset{x} \phi(x)$$
\end{proof}
\begin{rem} This is the only place where the full strength of assumption \eqref{eq-ve} is used. By restricting the predicates to which induction applies in \eqref{eq-ve}, it may therefore be possible to obtain a version of this interpretation which validates \formalsys{CZF} with a correspondingly restricted form of Set Induction, and this would have a predicative justification in the stricter sense of Sch\"utte and Feferman \cite{SF}. \end{rem}

\begin{lem}Extensional equality is reflexive, symmetric, and transitive.\end{lem}
\begin{proof} Symmetry is immediate from the fix-point definition of $\simeq$. Reflexivity and transitivity are shown by Set Induction. Suppose, for an inductive hypothesis, that each element of a v-set is equal to itself. It then follows from the fix-point definition of $\simeq$ that the v-set is equal to itself. Given v-sets $x$ and $z$, suppose, for a doubly inductive hypothesis, that for any $a$ in $x$ and $c$ in $z$, if there is a $b$ such that $a \simeq b \wedge b \simeq c$, then $a \simeq c$. If there is a $y$ such that $x \simeq y \wedge y \simeq z$ then for any $a$ in $x$, there is a $b$ in $y$ such that $a \simeq b$, and for that $b$ there is a $c$ in $z$ such that $b \simeq c$, so by hypothesis $a \simeq c$. Conversely for every $c$ in $z$ there is an $a$ in $x$ such that $c \simeq a$. So, $x \simeq z$. \end{proof}

\begin{lem}Extensional equality has the substitution property for sentences of first-order set theory.\end{lem}
\begin{proof} Proceed by induction on the structure of the formula. The base case is equality between two v-sets, and this is given by the previous lemma. For the propositional connectives and unbounded quantifiers it follows from intuitionistic predicate logic. What has to be considered is whether bounded quantification over extensionally equivalent domains is equivalent. So, given $x \simeq y$, take for inductive hypothesis that
$$\allin{u}{x} \allin{v}{y} u \simeq v \to (\phi(x,u) \leftrightarrow \phi(y,v))$$
It must be shown that 
$$(\somein{u}{x} \phi(x,u)) \leftrightarrow (\somein{v}{y} \phi(y,v))$$
and that
$$(\allin{u}{x} \phi(x,u)) \leftrightarrow (\allin{v}{y} \phi(y,v))$$
So, for the existential quantifier, suppose $\somein{u}{EL(x)} \phi(x,u)$. Now $x \simeq y$ so $\allin{u}{EL(x)}\somein{v}{EL(y)} u \simeq v$, so using the inductive hypothesis, it follows that $\somein{v}{EL(y)} \phi(y,v)$ as required. The other direction works the same way. For the universal quantifier, suppose $\allin{u}{EL(x)} \phi(x,u)$. Now $x \simeq y$, so $\allin{v}{EL(y)}\somein{u}{EL(x)} v \simeq u$, so again using the inductive hypothesis, it follows that $\allin{v}{EL(y)} \phi(y,v)$, as required, and again the other direction works the same way. \end{proof}

\begin{lem}Over formulas of first-order set theory, bounded quantification is equivalent to its usual definition in terms of unbounded quantification and the membership relation.\end{lem}
\begin{proof} Usually, they are defined as\begin{align*}
(\somein{y}{x} \phi(y)) &\iff (\someset{y} (y \in x) \wedge \phi(y)) \\
(\allin{y}{x} \phi(y)) &\iff (\allset{y} (y \in x) \to \phi(y))
\end{align*}
Whereas in this interpretation, the bounded quantifiers range over $EL$, and $y \in x$ means $\somein{y'}{EL(x)} y \simeq y'$. So the existential case becomes $\somein{y}{EL(x)} \phi(y)$ on the left-hand side and $\someset{y} \somein{y'}{EL(x)} y \simeq y' \wedge \phi(y)$. As noted above, $EL$ does not respect extensionality. Nevertheless, in the case where $\phi$ is a formula of first-order set theory, the previous lemma applies, together with the fact that the elements of a v-set are v-sets, to show the equivalence. Similar considerations apply to the universal case. \end{proof}

\begin{thm}\label{thm-g-pred} If $\mathbb{V}$ is taken as the domain of discourse, extensional equality is taken as the equality relation, and the membership relation is taken as the only other predicate, then intuitionistic predicate logic with equality is valid. \end{thm}
\begin{proof} Intuitionistic predicate logic is valid for number theory. Extensional equality has the correct properties for equality. The propositional connectives are interpreted as themselves. The unbounded quantifiers are simply quantifying over $\mathbb{N}$ and relativised to the inhabited predicate $\mathcal{V}$. Finally, the bounded quantifiers are equivalent to their usual definitions. \end{proof}

\section{Basic constructions}

To proceed, some temporary, trivial assumptions about $\mathcal{M}$ and $\mathcal{T}$ are made. These assumptions are subsumed in section 5. 

\begin{defn}\label{def-g-pi} $\top$ is a Meaningful condition which is True, and $\bot$ is a Meaningful condition which is not True. $\Pi(p,q)$ is a condition which asserts that both $p$ and $q$ are True; it is Meaningful if $p$ is a Meaningful condition, and if $p$ being True would imply that $q$ is a Meaningful condition. \end{defn}

\begin{thm}[Empty Set] \label{thm-g-empty} There is a v-set with no elements.\end{thm}
\begin{proof} It is given by:
$$ \empty = \Lambda n.\tuple{z,\bot} $$
$z$ is an irrelevant constant. \end{proof}

\begin{thm}[Pairing] \label{thm-g-pair} For any two v-sets, there is a v-set with those two v-sets as elements.\end{thm}
\begin{proof} If $x$ and $y$ are sets, their pairing can be given by: 
$$\{x,y\} = \Lambda n.\begin{cases}
  \tuple{x,\top}, & \text{if } n \text{ is even} \\ 
  \tuple{y,\top}, & \text{if } n \text{ is odd }
\end{cases}$$
\end{proof}

\begin{lem}\label{lem-g-part} Definition \ref{def-g-set} extends to computable partial sequences, as follows: suppose $x$ is a computable partial sequence such that, for all $n$, if $xn$ is defined then $\cd{x}{n}$ is a Meaningful condition, and if furthermore $\cd{x}{n}$ is True then $\el{x}{n}$ is a v-set. Then $\mathbf{D} x = \setbn{n}{\el{x}{n}}{\defd{xn} \wedge (\cd{x}{n} \text{ is True})}$ is a v-set.\end{lem}
\begin{proof} Define:\begin{align*}
T_0(x,n,u,p) &= \begin{cases} p & \text{if } \mathbf{T}(x,n,u) \\ \bot & \text{otherwise}\end{cases} \\
\mathbf{D}x &= \Lambda n.\tuple{\mathbf{U}(n_R)_L, T_0(x,n_L,n_R,\mathbf{U}(n_R)_R)}
\end{align*}
This will meet the requirements. For, \begin{align*}
\mathbf{D}x &\simeq \setbn{n}{\mathbf{U}(n_R)_L}{\mathcal{T}(T_0(x,n_L,n_R,\mathbf{U}(n_R)_R))}
\\&\simeq \setbn{n,u}{\mathbf{U}(u)_L}{\mathcal{T}(T_0(x,n,u,\mathbf{U}(u)_R))}
\\&\simeq \setbn{n,u}{\mathbf{U}(u)_L}{\mathbf{T}(x,n,u) \wedge \mathcal{T}(\mathbf{U}(u)_R)}
\\&\simeq \setbn{n}{(xn)_L}{\defd{xn} \wedge \mathcal{T}((xn)_R)}
\\&\simeq \setbn{n}{\el{x}{n}}{\defd{xn} \wedge \mathcal{T}(\cd{x}{n})}
\end{align*}
\end{proof}

\begin{thm}[Union] \label{thm-g-un} Given a v-set $x$, there is a v-set whose elements are the elements of the elements of $x$.\end{thm}
\begin{proof} The required v-set is given by:
$$\bigcup x = \mathbf{D}\Lambda n.\tuple{\el{\el{x}{n_L}}{n_R}, \Pi(\cd{x}{n_L},\cd{\el{x}{n_L}}{n_R})}$$
$y$ is an element of $\bigcup x$ if and only if there is an $n$ such that $\Pi(\cd{x}{n_L},\cd{\el{x}{n_L}}{n_R})$ is True and $y \simeq \el{\el{x}{n_L}}{n_R}$; this means there is an $n$ and an $m$ such that $\Pi(\cd{x}{n},\cd{\el{x}{n}}{m})$ is True and $y \simeq \el{\el{x}{n}}{m}$. By the Truth conditions of $\Pi$, this means that $\cd{x}{n}$ and $\cd{\el{x}{n}}{m}$ are True. The Truth of $\cd{x}{n}$ means that $\el{x}{n}$ an element of x, and the Truth of $\cd{\el{x}{n}}{m}$ means that $\el{\el{x}{n}}{m}$ is an element of $\el{x}{n}$. So, the elements of $\bigcup x$ are indeed the elements of the elements of $x$. \end{proof}
\begin{rem} The asymmetry in Definition \ref{def-g-pi} for $\Pi$ to be Meaningful is used here: $\cd{\el{x}{n}}{m}$ can only be assumed to be Meaningful when $\el{x}{n}$ is a v-set, which can only be assumed when $\cd{x}{n}$ is True. Indeed $\cd{\el{x}{n}}{m}$ can only be assumed to be defined when $\cd{x}{n}$ is True, so the Lemma \ref{lem-g-part} must apply. \end{rem}
\begin{thm}[Infinity] \label{thm-g-inf} There is an inhabited v-set such that each of its elements is also an element of one of its elements. \end{thm}
\begin{proof} The first infinite von Neumann ordinal, $\omega$, is such a v-set:
$$\omega = \Lambda n.\tuple{fn, \top}$$
where $f$ maps the natural number $n$ to the v-set $\neumarral{n}$, and can be given by general recursion: \begin{align*}
f &= \Lambda n.\Lambda m.\tuple{fm, A(m,n)} \\
A(m,n) &= \begin{cases} \top & \text{if } m<n \\ \bot & \text{if } m \ge n \end{cases}
\end{align*}
By ordinary induction (induction over $\mathbb{N}$), all the $\neumarral{n}$'s are v-sets, and therefore so is $\omega$.  
\end{proof}

\section{The red pill}

So far Extended Church's Thesis has not been invoked. The proofs of the previous sections are valid classically and constructively. To proceed to prove the classically unprovable principles of Replacement, Exponentiation, Dependent Choices, and Enough Projectives, it will be necessary to leave neutral territory. 

In constructive mathematics, Church's Thesis refers to the classically impossible assumption of Markov's school of constructivism that {\it all sequences are computable}. An even stronger form, \formalsys{ECT_0}, was introduced by Troelstra \cite{AT}. It is the schema:
$$(\allnat{n} \phi(n) \to \somenat{m} \psi(n,m)) \to \somenat{e} \allnat{n} \phi(n) \to \defd{en} \wedge \psi(n,en)$$
where $\phi$ is a so-called {\it almost-negative} sentence of arithmetic. $\psi$ is not constrained. The almost-negative sentences include the $\Sigma_1$ sentences (existential quantifiers over $\mathbb{N}$ directly in front of quantifier-free formulas of primitive recursive arithmetic) and are closed under conjunction, implication, and universal quantification over $\mathbb{N}$. Note that $\defd{en}$ is $\Sigma_1$, therefore almost-negative. 

Essentially, \formalsys{ECT_0} means that a statement of first-order arithmetic is true if and only if it has a computable witness, as given by the standard realisability clauses for arithmetic. The almost-negative sentences are essentially those which are already in the form ``$e$ witnesses $\phi$''. 

Here an assumption is required that will be called \formalsys{ECT_V}. It is the same schema as \formalsys{ECT_0}, but the class of almost-negative sentences allowed in the antecedent is expanded to include the new predicates $\mathcal{V}$, $\mathcal{M}$, and $\mathcal{T}$. The choice of defining assumptions for these predicates will make it possible to consider them almost-negative, as later shown by Meta-theorems \ref{mthm-c-ect} and \ref{mthm-cr-ect}.

The assumption that $\mathcal{T}$ is almost-negative means that only almost-negative sentences can be considered Meaningful. This is not entirely faithful to the informal reading of ``meaningful'', but from the point of view of set theory, this restriction results in no loss of generality, since \formalsys{ECT_V} allows any sentence to be converted into a single existential quantifier over $\mathbb{N}$ in front of an almost-negative sentence, and $\setbn{n}{fn}{\somenat{e}\phi(e,n)}$ is extensionally equal to $\setbn{n}{f(n_L)}{\phi(n_L,n_R)}$. (Actually Definition \ref{def-e-set} is an example of this, see section 9.)

The relation $\simeq$ cannot be considered almost-negative and so cannot appear in the antecedent of \formalsys{ECT_V}. Instead a new temporary predicate $\mathcal{R}$ is defined as a fix-point:
\begin{equation}\begin{matrix}
\hfill \allset{x,y} \mathcal{R}(e,x,y) & \leftrightarrow &(\allin{n}{IND(x)} \defd{e_Ln} & \wedge & (e_Ln)_L \in IND(y) \hfill
\\ &&&\wedge& \mathcal{R}((e_Ln)_R,\el{x}{n},\el{y}{(e_Ln)_L})) \hfill
\\ & \wedge & (\allin{n}{IND(y)} \defd{e_Rn} & \wedge & (e_Rn)_L \in IND(x) \hfill
\\ &&&\wedge& \mathcal{R}((e_Rn)_R,\el{y}{n},\el{x}{(e_Rn)_L})) \hfill
\end{matrix}\label{eq-rf}\end{equation}
Essentially $\mathcal{R}(e,x,y)$ says that $e$ witnesses $x \simeq y$. Lemma \ref{lem-g-eq} below shows that $\simeq$ can be defined in terms of $\mathcal{R}$, and so the assumption \eqref{eq-sim} is no longer needed. Here it will have to be assumed that $\mathcal{R}$ is almost-negative. This is justified in the next section where $\mathcal{R}$ will be defined in terms of $\mathcal{T}$.

Other than the two meta-theorems mentioned above, realisability will not be used explicitly; it will all be implicitly contained in \formalsys{ECT_V}. Also, the other principle used by Markov's school, namely Markov's principle, will not be used at all. Finally, the adjective ``computable'' is now vacuous and will be dropped. 

\begin{thm}[Strong Collection] \label{thm-g-coll} Suppose $\psi$ is an arbitrary binary relation (not necessarily a v-set), $x$ is a v-set, and to every element of $x$ there is at least one v-set related to it by $\psi$. Then there is a v-set which contains, for each element of $x$, at least one v-set related to it, and contains only such related v-sets. \end{thm}
\begin{proof} The hypothesis is that 
$$\allin{a}{x} \someset{b} \psi(a,b)$$
which means that 
$$\allin{n}{IND(x)} \someset{b} \psi(\el{x}{n},b)$$
Since $IND$ is defined in terms of $\mathcal{T}$, which is almost-negative, \formalsys{ECT_V} applies:
$$\somenat{e} \allin{n}{IND(x)} \defd{en} \wedge \mathcal{V}(en) \wedge \psi(\el{x}{n},en)$$
Define:
$$y = \mathbf{D}\Lambda n.\tuple{en, \cd{x}{n}}$$
If $\cd{x}{n}$ is True then $en$ is defined, so Lemma \ref{lem-g-part} applies. $y$ meets the requirements. For,
$$y \simeq \setb{en}{n \in IND(x) \wedge \defd{en}}$$
The elements of $y$ are equal to $en$ for some $n \in IND(x)$. By the properties of $e$, $\mathcal{V}(en)$ and $\psi(\el{x}{n},en)$, so $y$ is indeed a v-set, and each of its elements is related to some element of $x$ by $\psi$. Conversely by the properties of $e$, $en$ is in fact defined for all $n \in IND(x)$, so for any given element of $x$, $y$ does contain at least one v-set related to it.
\end{proof}
\begin{rem} \formalsys{ECT_V} in the above construction cannot guarantee $en$ is a v-set unless $\cd{x}{n}$ is True. Since the Truth of this condition is generally not even semi-decidable, $e$ cannot be patched to output only Valid sets. This is why Definition \ref{def-g-set} only requires $\el{x}{n}$ to be a v-set if $\cd{x}{n}$ is True. \end{rem}

\begin{cor}[Replacement] \label{cor-g-repl} Suppose $\psi$ is an arbitrary binary relation, $x$ is a v-set, and to every element of $x$ there is exactly one v-set related to it by $\psi$. Then there is a v-set which contains all and only the v-sets that are related to some element of $x$. \end{cor}

\begin{thm}[Subset Collection] \label{thm-g-scoll} Suppose $x$ and $y$ are v-sets, and $\psi$ is an arbitrary ternary relation. For any $c$, let $\psi_c$ denote the binary relation resulting from setting $c$ as the third argument of $\psi$. Then there is a v-set $z$ which collects all the $\psi(x)$-subsets of $y$, in the following sense: for any v-set $c$, if to every element of $x$ there is at least one element of $y$ related to it by $\psi_c$, then, $z$ contains a subset of $y$ which in turn contains, for every element of $x$, at least one element related to it by $\psi_c$, and contains only such related elements. \end{thm}
\begin{proof} Given v-sets $x$ and $y$, define: \begin{align*}
w &= \Lambda e.\mathbf{D}\Lambda n.\tuple{\el{y}{en}, \Pi(\cd{x}{n},\cd{y}{en})}\\
z &= \Lambda e.\tuple{we,\top}
\end{align*}
For any $e$, Lemma \ref{lem-g-part} applies, since $\cd{x}{n}$ is Meaningful, and if $en$ is defined then $\cd{y}{en}$ is also Meaningful, and its Truth implies that $\el{y}{en}$ is a v-sets. Therefore $we$ is always a v-set, so $z$ is a v-set as well. Moreover, $we$ is always a subset of $y$. Now, suppose $c$ is a v-set, and
$$\allin{a}{x}\somein{b}{y}\psi(a,b,c)$$
This means:
$$\allin{n}{IND(x)}\somein{m}{IND(y)}\psi(\el{x}{n},\el{y}{m},c)$$
And again \formalsys{ECT_V} gives the existence of an $e$ such that:
$$\allin{n}{IND(x)} \defd{en} \wedge en \in IND(y) \wedge \psi(\el{x}{n},\el{y}{m},c)$$
Now, \begin{align*}
we &\simeq \setbn{n}{\el{y}{en}}{\defd{en} \wedge \mathcal{T}(\cd{x}{n}) \wedge \mathcal{T}(\cd{y}{en})}
\\ &\simeq \setb{\el{y}{en}}{n \in IND(x) \wedge \defd{en} \wedge en \in IND(y)}
\end{align*}
By the properties of $e$, each element of $we$ is related by $\psi_c$ to some element of $x$, and conversely for each element of $x$, $we$ has an element which is related to it. In turn, $we$ is an element of $z$. So $z$ is a subset collection as required. \end{proof}

\begin{thm}[Enough Projectives] \label{thm-g-ep} Every v-set $x$ is the image of a v-set $y$ that is ``projective'', meaning that any v-set which is a binary relation whose domain is a superset of $y$, is itself the superset of a v-set that is a choice function with domain $y$. \end{thm}
\begin{proof} Given a v-set $x$, let $f = \Lambda n.\tuple{\stuple{\neumarral{n},\el{x}{n}},\cd{x}{n}}$ and $y = \Lambda n.\tuple{\neumarral{n},\cd{x}{n}}$. $f$ is a set-theoretic surjection from $y$ onto $x$. To see that $y$ is projective, suppose $g$ is a relation as in the hypothesis. Then:
$$\allin{a}{y}\someset{b} \stuple{a,b} \in g$$
And again, \formalsys{ECT_V} applies to give an $e$ such that:
$$\allin{n}{IND(y)} \defd{en} \wedge en \in \mathbb{V} \wedge \stuple{\el{y}{n},en} \in g$$
And furthermore $\cd{y}{n} = \cd{x}{n}$ and $\el{y}{n} = \neumarral{n}$. Define:
$$h = \mathbf{D}\Lambda n.\tuple{\stuple{\neumarral{n},\neumarral{en}},\cd{x}{n}}$$
If $\cd{x}{n}$ is True then $en$ is defined, so Lemma \ref{lem-g-part} applies. All the $\neumarral{n}$'s are distinct, so $h$ is in fact a function. Moreover it is a subset of $g$, and its domain is $y$. It is therefore a choice function as required. \end{proof}

\begin{thm}[Dependent Choices]\label{thm-g-dc} Suppose $z$ is a v-set, $\psi$ is an arbitrary binary relation, and to every element of $z$ there is an element of $z$ related to it by $\psi$. Then to every $x \in z$ there is a v-set which is a choice sequence of elements of $z$ (function with domain $\omega$ and range a subset of $z$) starting with $x$ and such that each successor is related to its predecessor. \end{thm}
\begin{proof} It can be shown to follow from Enough Projectives \cite[\S 8.3]{AR}, but here, it is simple enough to construct directly. By hypothesis, $\allin{u}{z} \somein{v}{z} \psi(u,v)$; apply \formalsys{ECT_V} to get an $e$ such that:
$$\allin{n}{IND(z)} \defd{en} \wedge en \in IND(z) \wedge \psi(\el{z}{n},\el{z}{en})$$
Given $x \in z$ there is an $n_0 \in IND(z)$ such that $\el{z}{n_0} \simeq x$. Define $f$ by primitive recursion such that $f0 = n_0$ and $f(n+1) = e(fn)$. Then define:
$$g = \Lambda n.\tuple{\stuple{\neumarral{n},\el{z}{fn}},\top}$$
This v-set is the required choice sequence. \end{proof}

\begin{lem} \label{lem-g-eq} For any two v-sets $x$ and $y$, $x \simeq y$ if and only if there is an $e$ such that $\mathcal{R}(e,x,y)$. \end{lem}
\begin{proof} Given some $e$: \begin{align*}
\mathcal{R}(e,x,y) \leftrightarrow &(\allin{n}{IND(x)} \defd{e_Ln} \wedge (e_Ln)_L \in IND(y) \wedge \mathcal{R}((e_Ln)_R,\el{x}{n},\el{y}{(e_Ln)_L})) 
\\ \wedge & (\allin{n}{IND(y)} \defd{e_Rn} \wedge (e_Rn)_L \in IND(x) \wedge \mathcal{R}((e_Rn)_R,\el{y}{n},\el{x}{(e_Rn)_L}))
\end{align*}
Take, for a doubly set-inductive hypothesis, that for any element $a$ of $x$ and any element $b$ of $y$, $a \simeq b$ if and only if there is an $f$ such that $\mathcal{R}(f,a,b)$. So, if $\somenat{e} \mathcal{R}(e,x,y)$ it follows that:
$$(\allin{n}{IND(x)} \somein{m}{IND(y)} \el{x}{n} \simeq \el{y}{m}) \wedge (\allin{m}{IND(y)} \somein{n}{IND(x)} \el{y}{m} \simeq \el{x}{n})$$
And conversely, by \formalsys{ECT_V}, the above implies $\somenat{e} \mathcal{R}(e,x,y)$. But the above means that
$$(\allin{a}{x} \somein{b}{y} a \simeq b) \wedge (\allin{b}{y} \somein{a}{x} b \simeq a)$$
Which, by \eqref{eq-sim}, just means
$$x \simeq y$$
So, by Set Induction, 
$$x \simeq y \leftrightarrow \somenat{e} \mathcal{R}(e,x,y)$$
\end{proof}

\begin{lem}\label{lem-g-real} For any sentence $\phi$ expressible in the language of first-order set theory, there is an almost-negative predicate $\phi'$ such that $\phi \iff \somenat{e} \phi'(e)$.\end{lem}
\begin{proof} This transformation is effected recursively on the structure of $\phi$. \begin{itemize}
\item The transformation of the base case $x \simeq y$ is given by the previous lemma: $\mathcal{R}(e,x,y)$.
\item The other base case, $\bot$, transforms to itself.
\item Conjunction, disjunction, and the existential quantifiers are straightforward: \begin{align*}
(\somenat{e} \phi'(e)) \wedge (\somenat{e} \psi'(e)) &\leftrightarrow \somenat{e} (\phi'(e_L) \wedge \psi'(e_R)) \\
(\somenat{e} \phi'(e)) \vee (\somenat{e} \psi'(e)) &\leftrightarrow \somenat{e} (e_L = 0 \to \phi'(e_R)) \wedge (e_L \ne 0 \to \psi'(e_R)) \\
(\somein{x}{y} \somenat{e} \phi'(e,x)) &\leftrightarrow \somenat{e} \mathcal{T}(\cd{y}{e_L}) \wedge \phi'(e_R,\el{y}{e_L}) \\
(\someset{x} \somenat{e} \phi'(x)) &\leftrightarrow \somenat{e} \mathcal{V}(e_L) \wedge \phi'(e_R, e_L)
\end{align*}
\item Implication and the universal quantifiers require \formalsys{ECT_V} for their transformation:\begin{align*}
((\somenat{e} \phi'(e)) \to (\somenat{e} \psi'(e))) &\leftrightarrow \somenat{e} \allnat{f} \phi'(f) \to \defd{ef} \wedge \psi'(ef) \\
(\allin{x}{y} \somenat{e} \phi'(e,x)) &\leftrightarrow \somenat{e} \allnat{n} \mathcal{T}(\cd{y}{n}) \to \defd{en} \wedge \phi'(en,\el{y}{n}) \\
(\allset{x} \somenat{e} \phi'(e,x)) &\leftrightarrow \somenat{e} \allnat{x} \mathcal{V}(x) \to \defd{ex} \wedge \phi'(ex,x)
\end{align*}
\end{itemize}
The resulting formulas are in the right form, given that $\mathcal{R}$, $\mathcal{T}$, and $\mathcal{V}$ are almost-negative.
\end{proof}

\begin{thm}[Relativised Dependent Choices] \label{thm-g-rdc} Suppose $\phi$ is a unary predicate expressible in the language of first-order set theory (i.e. a class), $\psi$ is an arbitrary binary relation (not necessarily a class), and to each $x$ satisfying $\phi$ there is at least one $y$ satisfying $\phi$ that is related to it by $\psi$. Then, to every $x$ satisfying $\phi$ there is a choice sequence of v-sets, all of which satisfy $\phi$, starting with $x$ and such that each successor is related to its predecessor by $\psi$. \end{thm}
\begin{proof} The hypothesis is that:
$$\allset{x} \phi(x) \to \someset{y} \phi(y) \wedge \psi(x,y)$$
$\phi$ is not necessarily in a form to which \formalsys{ECT_V} can be applied. The previous lemma is used to transform the hypothesis into:
$$\allset{x} (\somenat{f} \phi'(f,x)) \to \someset{y} (\somenat{g} \phi'(g,y)) \wedge \psi(x,y)$$
Applying \formalsys{ECT_V} to this transformed hypothesis results in an $e$ such that:\begin{align*}
\allset{x}\allnat{f} \phi'(f,x) \to &\defd{e\tuple{x,f}} \wedge \mathcal{V}((e\tuple{x,f})_L) 
\\ &\wedge \phi'((e\tuple{x,f})_R,(e\tuple{x,f})_L) \wedge \psi(x,(e\tuple{x,f})_L)
\end{align*}
Now, given a v-set $x$ such that $\phi(x)$, there is an $f$ such that $\phi'(f,x)$. So define a sequence $j$ by primitive recursion on  $e$: \begin{align*}
j0 &= \tuple{x,f} \\
j(n+1) &= e(jn)
\end{align*}
The actual set-theoretic choice sequence is given by:
$$k = \Lambda n.\tuple{\stuple{\neumarral{n},(jn)_L}, \top}$$
\end{proof}

\section{The missing piece}

One important principle of \formalsys{CZF} is missing: Bounded Separation. For this, a non-trivial assumption on Meaning is needed. Under this assumption it will be shown that $\mathcal{R}$ can be defined in terms of $\mathcal{T}$, so assumption \eqref{eq-rf} from the previous section can be dropped. It will also be shown that Definition \ref{def-g-pi} is redundant. This will leave just first-order arithmetic, the predicates $\mathcal{V}$, $\mathcal{M}$, $\mathcal{T}$, the schema \formalsys{ECT_V}, the assumptions \eqref{eq-vi} and \eqref{eq-ve}, and two more assumptions below, \eqref{eq-xit} and \eqref{eq-xim}.

So far only closure under conjunction was required. To proceed, one may expect that Meaning will require closure under the remaining propositional connectives and quantifiers. As discussed in the previous section, only almost-negative sentences need be considered Meaningful. $\Sigma_1$ sentences are therefore taken to be Meaningful, and Meaning is taken to be closed under implication and universal quantification. The latter means that given a sequence of Meaningful conditions, there is another Meaningful condition expressing their infinitary conjunction. To simplify the formal treatment, these are combined into a single operator.

\begin{defn}\label{def-g-xi} $\Xi(f,g)$ is a condition which asserts that, for all $n$, whenever $fn$ is defined and True, $gn$ is also defined and True. It is Meaningful if, for all $n$, when $fn$ is defined, it is Meaningful, and when $fn$ and $gn$ are defined and $fn$ is True, $gn$ is Meaningful. \end{defn}

It is not hard to see that this captures all of the above, including the asymmetrical Meaning conditions on $\Pi$ of Definition \ref{def-g-pi}. Now, formal assumptions on $\mathcal{M}$ and $\mathcal{T}$ are made following Definition \ref{def-g-xi}. Assume:\begin{align}
\allnat{f,g} &\mathcal{M}(\Xi(f,g)) \to (\mathcal{T}(\Xi(f,g)) \leftrightarrow (\allnat{n} \defd{fn} \wedge \mathcal{T}(fn) \to \defd{gn} \wedge \mathcal{T}(gn))) \label{eq-xit}\\
\allnat{f,g} &(\allnat{n} \defd{fn} \to \mathcal{M}(fn)) \wedge (\allnat{n} \defd{fn} \wedge \defd{gn} \wedge \mathcal{T}(fn) \to \mathcal{M}(gn)) \to \mathcal{M}(\Xi(f,g))  \label{eq-xim}
\end{align}

\begin{thm}\label{thm-g-mnm} ``Meaningful'' is not a Meaningful adjective. That is, there is no $\mu$ such that $\mu \tuple{n,m}$ is defined and Meaningful for all $n$ and $m$, and such that $n$ is Meaningful if and only if there exists an $m$ such that $\mu \tuple{n,m}$ is True.\end{thm}
\begin{proof} First define:
$$\nu = \Lambda n.\Xi(\Lambda m.n, \Lambda m.\bot)$$
This is just a negation operator. $\nu n$ is Meaningful iff $n$ is, and it is True iff $n$ is not True. Now, if there were a $\mu$ satisfying the hypothesis, then an extended negation operator could be defined:
$$\tilde{\nu} = \Lambda n. \Xi(\Lambda m. \mu \tuple{n,m}, \Lambda m.\nu n)$$
By hypothesis, $\mu \tuple{n,m}$ would be Meaningful for all $n$ and $m$, and if $\mu \tuple{n,m}$ is True then $n$ is Meaningful. So, by the Meaning conditions for $\Xi$, $\tilde{\nu}n$ would be Meaningful for all $n$. It would be True iff $n$ being Meaningful implied that $n$ is not True. Now define by general recursion: 
$$\mathbf{L} = \tilde{\nu}\mathbf{L}$$
And this is a contradiction, because $\mathbf{L}$ would be Meaningful, and $\mathbf{L}$ would be True iff $\mathbf{L}$ were not True. So, there is no such $\mu$. \end{proof}

\begin{lem}\label{lem-g-req} Given a number $e$ and v-sets $x$ and $y$, there is a Meaningful condition $\overline{\mathcal{R}}(e,x,y)$ which can be constructed using only the operator $\Xi$ and such that $\mathcal{T}(\overline{\mathcal{R}}(e,x,y)) \leftrightarrow \mathcal{R}(e,x,y)$.\end{lem}
\begin{proof} Define $\overline{\mathcal{R}}$ by general recursion: \begin{align*}
\overline{\mathcal{R}}(e,x,y) &= \Pi(H(e_L,x,y),H(e_R,y,x)) \\
H(e,x,y) &= \Xi(\Lambda n.\cd{x}{n}, \Lambda n.J(n,(en)_L,(en)_R,x,y)) \\
J(n,m,e,x,y) &= \Pi(\cd{y}{m}, \overline{\mathcal{R}}(e,\el{x}{n},\el{y}{m}))
\end{align*}
$\Pi$ can of course be expressed in terms of $\Xi$.

Now, $H(e,x,y)$ is always defined, so $\overline{\mathcal{R}}(e,x,y)$ is too. Given v-sets $x$ and $y$, take, for a doubly set-inductive hypothesis, that $\overline{\mathcal{R}}(f,a,b)$ is Meaningful for all numbers $f$, all elements $a$ of $x$, and all elements $b$ of $y$. $\cd{x}{n}$ is always Meaningful; if true, then $\el{x}{n}$ is a v-set and an element of x. The same goes for $\cd{y}{m}$ are $\el{y}{m}$. By the inductive hypothesis, $H(e,x,y)$ is Meaningful, and so is $\overline{\mathcal{R}}(e,x,y)$. So by Set Induction $\overline{\mathcal{R}}$ is Meaningful for all v-sets. 

Now take for inductive hypothesis that $\overline{\mathcal{R}}(f,a,b)$ has the right Truth conditions for all numbers $f$, all elements $a$ of $x$, and all elements $b$ of $y$. The Truth conditions of $\overline{\mathcal{R}}(e,x,y)$ expand to: \begin{align*}
\mathcal{T}(\overline{\mathcal{R}}(e,x,y)) \leftrightarrow &(\allin{n}{IND(x)} \defd{e_Ln} \wedge (e_Ln)_L \in IND(y) \wedge \mathcal{T}(\overline{\mathcal{R}}((e_Ln)_R,\el{x}{n},\el{y}{(e_Ln)_L})))
\\ \wedge & (\allin{n}{IND(y)} \defd{e_Rn} \wedge (e_Rn)_L \in IND(x) \wedge \mathcal{T}(\overline{\mathcal{R}}((e_Rn)_R,\el{y}{n},\el{x}{(e_Rn)_L})))
\end{align*}
This matches \eqref{eq-rf}, so, by Set Induction, $\overline{\mathcal{R}}$ has the right Truth conditions. \end{proof}

\begin{lem}[Kronecker Delta] \label{lem-g-kron} For any two v-sets $x$ and $y$, there is a v-set which is a subset of $\neumarral{1}$ and which is inhabited if and only if $x \simeq y$. \end{lem}
\begin{proof} It is given by
$$\delta_{xy} = \Lambda n.\tuple{\neumarral{0}, \overline{\mathcal{R}}(n,x,y)}$$
It follows from Lemmas \ref{lem-g-eq} and \ref{lem-g-req} that $\neumarral{0} \in \delta_{xy} \leftrightarrow x \simeq y$, and clearly $\neumarral{0}$ is the only possible element of $\delta_{xy}$, so $\delta_{xy} \subseteq \neumarral{1}$. \end{proof}

\begin{lem}[Infimum] \label{lem-g-inf} For any v-set $x$ whose elements are all subsets of $\neumarral{1}$, there is a v-set which is a subset of $\neumarral{1}$ and which is inhabited if and only if all the elements of $x$ are inhabited. \end{lem}
\begin{proof} A subset of $\neumarral{1}$ is inhabited if and only if it is extensionally equal to $\neumarral{1}$. Use Replacement to form a v-set $\tilde{x}$ in which all the elements of $x$ are replaced by $\neumarral{1}$. Then $x \simeq \tilde{x}$ if and only if all elements of $x$ are inhabited. So, using the previous lemma, $\delta_{x \tilde{x}}$ is the required infimum. \end{proof}

\begin{thm}[Bounded Separation] \label{thm-g-sep} Suppose $y$ is a v-set, and $\phi$ is a unary predicate which can be expressed in the language of first-order set theory with only bounded quantifiers (that is, the quantifiers $\somein{x}{y}$ and $\allin{x}{y}$). Then there is a v-set containing all and only the elements of $y$ which satisfy $\phi$. \end{thm}
\begin{proof} It can be shown that this follows from Extensionality, Empty Set, Pairing, Union, Replacement, Kronecker Delta, and Infimum \cite[\S 3.3]{AR}. (And in fact Infimum follows from the others, as shown above.) \end{proof}

\begin{thm}[Exponentiation] For all v-sets $x$ and $y$, there is a v-set ${}^xy$ of all functions from $x$ to $y$.\end{thm}
\begin{proof} Define the relation $\psi$ as
$$\psi(a,b,c) \iff b \in c \wedge \someset{d} b\simeq\stuple{a,d}$$
If $c$ is a function from $x$ to $y$, then every element of $x$ is related to exactly one element of $x \times y$ by $\psi_c$, and the collection of all such elements of $x \times y$ is $c$ itself. Using Subset Collection it follows that there is a v-set which is a superset of ${}^xy$. The property of being a function from $x$ to $y$ can be expressed by a bounded formula, therefore Bounded Separation applies to get exactly ${}^xy$. \end{proof}

\section{Six impossible theorems before breakfast}

The results from the previous sections are classically and constructively valid. They are all theorems of \formalsys{CZF+RDC+EP}, which in turn is a sub-theory of \formalsys{ZFC}. Conversely, they axiomatise \formalsys{CZF+RDC+EP}. As remarked in section 1, adding the principle of the Excluded Middle has the following effect: \formalsys{CZF+EM=ZF}, \formalsys{CZF+RDC+EM=ZF+DC}, \formalsys{CZF+RDC+EP+EM=ZF+EP}; not much is known about the latter, other than that \formalsys{ZF+DC < ZF+EP \le ZFC}.

Though the results are valid, the proofs of these results are, in sections 4 and 5, classically invalid, because \formalsys{ECT} contradicts \formalsys{EM}. \formalsys{ECT} should be expected to have set-theoretic consequences that are incompatible with \formalsys{ZFC}. In this section, a few basic ones are mentioned. They are all fairly straightforward consequences of the first theorem, \formalsys{ESP}. As will be shown in the following section, \formalsys{ESP} is conservative over \formalsys{CZF+RDC}. This is in contrast with \formalsys{EM}, which is highly non-conservative in this setting: \formalsys{ZF} proves the consistency of \formalsys{CZF}, indeed, \formalsys{CZF} is provably consistent, by transfinite induction up to the Bachmann-Howard ordinal, while \formalsys{ZF} is far beyond any system for which a constructive consistency proof is known \cite{PML2}.

\begin{thm}[Enough Subcountable Projectives] \label{thm-g-esp} Every v-set is the image of a v-subset of $\omega$ which is also projective.\end{thm}
\begin{proof} Actually, in the construction of Theorem \ref{thm-g-ep}, the projective v-set was already a subset of $\omega$. \end{proof}

\begin{cor}[Subcountability] \label{cor-g-sub} Every v-set is subcountable, that is, the image of a v-subset of $\omega$.\end{cor}

\begin{thm}\label{thm-g-nopow}The Power Set principle is false. In particular, $\pow{\omega}$, the class of all v-subsets of $\omega$, is a proper class. \end{thm}
\begin{proof} This is an adaptation of Cantor's diagonal theorem. Suppose $\pow{\omega}$ were a v-set. By the above, it would be subcountable. There would be an $x \subseteq \omega$ and a surjective $f: x \twoheadrightarrow \mathbb{P}(\omega)$. The ``Cantor diagonal set'' $\mathbf{C} = \setb{y \in x}{y \not\in f(y)}$ would be a v-set, by Bounded Separation, and $\mathbf{C} \subseteq x \subseteq \omega$. So by $f$'s surjectivity, $\mathbf{C} \simeq f(z)$ for some $z \in x$, and it would follow that $z \in \mathbf{C} \leftrightarrow z \not\in \mathbf{C}$. Contradiction. \end{proof}
\begin{rem} On the other hand, ${}^\omega\omega$ is a v-set, by Exponentiation. It can be shown, by a different adaptation of Cantor's diagonal theorem, that ${}^\omega\omega$ is not countable. But it is subcountable. \end{rem}

\begin{cor}\label{cor-g-p1}Even $\p1$ is a proper class.\end{cor}
\begin{proof} Suppose $\p1$ were a v-set. Then for any v-set $x$, ${}^x\p1$ would be a v-set, by Exponentiation. But, ${}^x\p1$ is isomorphic to $\mathbb{P}(x)$. Specifically $\mathbb{P}(x) = \setb{\setb{y \in x}{\neumarral{0} \in f(y)}}{f \in {}^x\p1}$; the inner set is a v-set by Bounded Separation, and then $\mathbb{P}(x)$ would be a v-set by Replacement. But $\omega$ is a v-set and $\mathbb{P}(\omega)$ is not, so $\p1$ is not either. \end{proof}
\begin{rem} A proof giving a different perspective is that if $\p1$ were a v-set, then ``Meaningful'' would be Meaningful. That is, from a v-set $x$ and a witness that $x$ has the properties of $\p1$, a $\mu$ which violates Theorem \ref{thm-g-mnm} can be constructed. This construction is omitted here. \end{rem}

\begin{thm}\label{thm-g-noac}The unrestricted principle of Choice is false.\end{thm}
\begin{proof} This is an adaptation of Diaconescu's result \cite{RD}. Given any $x \subseteq \neumarral{1}$, define:\begin{align*}
A &= \setb{n\in\neumarral{2}}{n \simeq \neumarral{0} \vee (n \simeq \neumarral{1} \wedge \neumarral{0} \in x)} \\
B &= \setb{n\in\neumarral{2}}{n \simeq \neumarral{1} \vee (n \simeq \neumarral{0} \wedge \neumarral{0} \in x)}
\end{align*} and these are v-sets by Bounded Separation. Applying the principle of Choice to $\{A,B\}$, there would be a choice function $f$ such that $f(A) \in A \wedge f(B) \in B$. By definition of these two v-sets, $(f(A) \simeq \neumarral{0} \vee (f(A) \simeq \neumarral{1} \wedge \neumarral{0} \in x)) \wedge (f(B) \simeq \neumarral{1} \vee (f(B) \simeq \neumarral{0} \wedge \neumarral{0} \in x))$, from which it would follow that $f(A) \not\simeq f(B) \vee \neumarral{0} \in x$. But if $\neumarral{0} \in x$ then by the definitions, $A \simeq B$, and by Extensionality, $f(A) \simeq f(B)$. So $f(A) \not\simeq f(B) \to \neumarral{0} \not\in x$, and as a result $\neumarral{0} \not\in x \vee \neumarral{0} \in x$. So $x \simeq \neumarral{0} \vee x \simeq \neumarral{1}$, but $x$ was an arbitrary element of $\p1$. So, $\p1 \simeq \{\neumarral{0},\neumarral{1}\} \simeq \neumarral{2}$. But $\neumarral{2}$ is a v-set, while $\p1$ is not. Contradiction. \end{proof}

\begin{thm}\label{thm-g-noreg}The Foundation principle is false: not all inhabited v-sets are disjoint from one of their elements.\end{thm}
\begin{proof} Suppose all inhabited v-sets were disjoint from one of their elements. Consider again any $x \subseteq \neumarral{1}$ and the corresponding set $B$ defined in the previous proof. It is inhabited and would be disjoint from one of its elements, $n$. By definition of $B$, $n \simeq \neumarral{1}$ or $n \simeq \neumarral{0} \wedge \neumarral{0} \in x$. In the former case, if $B$ were disjoint from $\neumarral{1}$, meaning $\neumarral{0} \not\in B$, then $\neumarral{0} \not\in x$. In the latter case of course $\neumarral{0} \in x$. So again, $\neumarral{0} \not\in x \vee \neumarral{0} \in x$, and $x$ was arbitrary, so $\p1 \simeq \neumarral{2}$. Contradiction. \end{proof}
\begin{rem} As mentioned in section 1, a contrapositive form of Foundation is a theorem: no inhabited v-set intersects all of its elements.\end{rem}

\section{CZF through the looking glass}

The assumptions \eqref{eq-vi} and \eqref{eq-ve} are a standard accessibility definition. The assumptions \eqref{eq-xit} and \eqref{eq-xim} are a special case of inductive-recursive definition: here, the Meaningful sentences are being inductively generated, but this depends on Truth, which is defined recursively at the same time. However, no induction principle corresponding to \eqref{eq-xim} is assumed, so by themselves \eqref{eq-xit} and \eqref{eq-xim} are actually a very weak form of inductive-recursive definition. This is roughly analogous to the simple universe construction in Martin-L\"of type theory \cite{PML}. It could likely be given a predicative justification in the stricter sense of Sch\"utte and Feferman \cite{SF}, though this is not attempted here.

In any case, such definitions are within the power of \formalsys{CZF}. To interpret $\mathcal{M}$ and $\mathcal{T}$ set-theoretically, define the following classes: \begin{align*}
\Phi &= \setb{\stuple{M,a}}{\somein{f,g}{\omega} a \simeq \stuple{\Xi(f,g),Z(f,g,M)} \wedge f \in X(M) \wedge \stuple{f,g} \in Y(M)} \\
X(M) &= \setb{f \in \omega}{\allin{n}{\omega} \defd{fn} \to \someset{h} \stuple{fn,h} \in M} \\
Y(M) &= \setb{\stuple{f,g} \in \omega^2}{\allin{n}{\omega} \defd{fn} \wedge \defd{gn} \wedge (\stuple{fn,\neumarral{1}} \in M) \to \someset{h} \stuple{gn,h} \in M} \\
Z(f,g,M) &= \setb{\neumarral{0}}{\allin{n}{\omega} (\defd{fn} \wedge \stuple{fn,\neumarral{1}} \in M) \to (\defd{gn} \wedge \stuple{gn,\neumarral{1}} \in M)}
\end{align*}
The Class Inductive Definition Theorem can be proven in \formalsys{CZF} \cite[\S 5]{AR}, by which the least $\Phi$-closed class, $M^\infty$, can be formed. ``$\Phi$-closed'' means that if $M$ is a subset of $M^\infty$, and $\stuple{M,a} \in \Phi$, then $a \in M^\infty$; $M^\infty$ is a subclass of any class with this property.

Take $\mathcal{M}(x)$ to mean $\someset{h}\stuple{x,h}\in M^\infty$ and $\mathcal{T}(x)$ to mean $\stuple{x,\neumarral{1}} \in M^\infty$. The interpretation of $\mathcal{V}$ comes easily. Define the classes:\begin{align*}
\Psi &= \setb{\stuple{V,x}}{x \in P \cap Q(V)} \\
P &= \setb{x \in \omega}{\allin{n}{\omega} \defd{xn} \wedge \someset{h} \stuple{\cd{x}{n},h} \in M^\infty} \\
Q(V) &= \setb{x \in \omega}{\allin{n}{\omega} \defd{xn} \wedge (\stuple{\cd{x}{n},\neumarral{1}} \in M^\infty \to \el{x}{n} \in V)}
\end{align*}
Form the least $\Psi$-closed class, $V^\infty$. Take $\mathcal{V}(x)$ to mean $x \in V^\infty$.

\begin{lem}\label{lem-c-con}$M^\infty$ is coherent. That is, if there are $x$, $h$, and $h'$ such that $\stuple{x,h} \in M^\infty$ and $\stuple{x,h'} \in M^\infty$, then $h \simeq h'$.\end{lem}
\begin{proof} Define the coherent subclass of $M^\infty$ as follows:
$$M' = \setb{\stuple{x,h} \in M^\infty}{\allin{h'}{\p1} \stuple{x,h'} \in M^\infty \to h \simeq h'}$$
$M'$ will in fact be closed under the same operator $\Phi$ that defined $M^\infty$ above. To see this, suppose $M$ is a subset of $M'$, and $a$ is a set such that $\stuple{M,a} \in \Phi$. Then, by definition of $\Phi$, there are $f,g \in \omega$ such that $a \simeq \stuple{\Xi(f,g),Z(f,g,M)}$. $Z(f,g,M)$ is a valid set because $M$ is, and $Z$ is otherwise given by a bounded formula. Furthermore, $f \in X(M)$, so for all $n$, if $fn$ is defined then there is a $j \in \p1$ such that $\stuple{fn,j} \in M$. $M$ is a subset of $M'$, so by definition of $M'$, this $j$ is unique (up to extensional equality). Similarly $\stuple{f,g} \in Y(M)$ so for all $n$, if $fn$ and $gn$ are defined and $\stuple{fn,\neumarral{1}} \in M$ then there is a unique $k \in \p1$ such that $\stuple{gn,k} \in M$. 

Now $a \in M^\infty$ since $M \subseteq M^\infty$ and $\stuple{M,a} \in \Phi$. Suppose there is another $\stuple{\Xi(f,g),h'} \in M^\infty$. Then there exists a set $N \subseteq M^\infty$ such that $f \in X(N)$, $\stuple{f,g} \in Y(N)$ and $h' \simeq Z(f,g,N)$. By definition of $X$, for all $n$, if $fn$ is defined there is a $j' \in \p1$ such that $\stuple{fn,j'} \in N$. But these are the same $fn$'s as above, and the corresponding $j$'s are unique, so $j' \simeq j$. There is a similar argument for $Y$. Therefore $Z(f,g,M) \simeq Z(f,g,N)$, which means that $h \simeq h'$. 

So, $h$ being unique, $a \in M'$, and the class $M'$ is $\Phi$-closed. But $M^\infty$ is the least such class, so all of $M^\infty$ is coherent. \end{proof}

\begin{lem}\label{lem-c-xi} In $M^\infty$, \eqref{eq-xit} and \eqref{eq-xim} are true.\end{lem}
\begin{proof} Suppose $f \in X(M^\infty)$ and $\stuple{f,g} \in Y(M^\infty)$. By the previous lemma the corresponding $\stuple{fn,j} \in M^\infty$ and $\stuple{gn,k} \in M^\infty$ are unique, so by Replacement and Bounded Separation the following are sets:\begin{align*}\
M_X &= \setb{\stuple{fn,j} \in M^\infty}{n \in \omega \wedge \defd{fn}} \\
M_Y &= \setb{\stuple{gn,k} \in M^\infty}{n \in \omega \wedge \defd{fn} \wedge \defd{gn} \wedge \stuple{fn,\neumarral{1}} \in M_X} \\
h &= Z(f,g,M_X \cup M_Y)
\end{align*}
$f \in X(M_X \cup M_Y)$ and $\stuple{f,g} \in Y(M_X \cup M_Y)$, and $M^\infty$ is $\Phi$-closed, so $\stuple{\Xi(f,g),h} \in M^\infty$, so $\Xi(f,g)$ is Meaningful as required by \eqref{eq-xim}. Finally, $Z$ expresses the correct Truth conditions for \eqref{eq-xit}. \end{proof}

\begin{lem}\label{lem-c-set} In $V^\infty$ and $M^\infty$, \eqref{eq-vi} and \eqref{eq-ve} are true.\end{lem}
\begin{proof} $P$ and $Q$ directly encode the conditions expressed by \eqref{eq-vi}; \eqref{eq-ve} follows from the minimality of $V^\infty$. \end{proof}

\begin{mthm} \label{mthm-c-ect} Assumptions \eqref{eq-vi}, \eqref{eq-ve}, \eqref{eq-xit}, and \eqref{eq-xim}, along with \formalsys{ECT_V}, are true in a realisability interpretation, provided that the four assumptions hold (without \formalsys{ECT_V}) in the underlying meta-theory. \end{mthm}
\begin{proof} The realisability interpretation is simply the standard one for first-order arithmetic, with the additional clauses that witnesses for $\mathcal{V}$, $\mathcal{M}$, and $\mathcal{T}$ are trivial, that is:\begin{align*}
(e \Vdash \mathcal{V}(x)) &\iff \mathcal{V}(x) \\
(e \Vdash \mathcal{M}(x)) &\iff \mathcal{M}(x) \\
(e \Vdash \mathcal{T}(x)) &\iff \mathcal{T}(x)
\end{align*}
The proof of \formalsys{ECT_0} is standard \cite{AT}, and with these clauses it extends to \formalsys{ECT_V} trivially. Moreover, since $\mathcal{V}$, $\mathcal{M}$, and $\mathcal{T}$ are almost-negative, it immediately follows that \eqref{eq-vi}, \eqref{eq-xit}, and \eqref{eq-xim} are almost-negative. They are simply true if they are true in the meta-theory. 

What remains, then, is \eqref{eq-ve}. This is a schema, and for some instances the witness needs to do something non-trivial, but one can be given whose validity is proven in terms of instances of the same schema in the meta-theory. Define by general recursion: 
$$\rho = \Lambda e.\Lambda x. (e'x)(\Lambda n.(\rho e)(\el{x}{n}))$$
Here $e'$ is a simple syntactic transformation introduced below for clarity of presentation. It will follow that $(\rho e)x$ iterates $e$ by set recursion along $x$, and as a result, $\rho$ witnesses \eqref{eq-ve} for any $\phi$, provided \eqref{eq-ve} holds in the meta-theory. 

Suppose $e$ witnesses the antecedent of \eqref{eq-ve}:
$$e \Vdash \allnat{x}(\allnat{n} \defd{xn} \wedge \mathcal{M}(\cd{x}{n}) \wedge (\mathcal{T}(\cd{x}{n}) \to \phi(\el{x}{n}))) \to \phi(x)$$
This means that:\begin{align*}
\allnat{x,a} &(\allnat{n} \defd{xn} \wedge \mathcal{M}(\cd{x}{n}) \wedge (\mathcal{T}(\cd{x}{n}) \to \defd{an} \wedge (an \Vdash \phi(\el{x}{n})))) \\ & \to \defd{(e'x)a} \wedge (e'x)a \Vdash \phi(x)
\end{align*}
The syntactic transformation $e \mapsto e'$ is chosen so that $e$ ignores the trivial witnesses for $\mathcal{M}$, $\mathcal{T}$, and $\defd{xn}$. Given a set $x$, suppose that by way of inductive hypothesis:
$$\allin{a}{x} \defd{(\rho e)a} \wedge (\rho e)a \Vdash \phi(a)$$
This means that
$$\allnat{n} \mathcal{T}(\cd{x}{n}) \to \defd{(\rho e)(\el{x}{n})} \wedge (\rho e)(\el{x}{n}) \Vdash \phi(\el{x}{n})$$
And since $x$ is Valid, 
$$\allnat{n} \defd{xn} \wedge \mathcal{M}(\cd{x}{n})$$
So, substituting $\Lambda n.(\rho e)(\el{x}{n})$ for $a$ in the assumption on $e$, it follows that
$$\defd{(e'x)(\Lambda n.(\rho e)(\el{x}{n}))} \wedge (e'x)(\Lambda n.(\rho e)(\el{x}{n})) \Vdash \phi(x)$$
By definition of $\rho$,
$$\defd{(\rho e)x} \wedge (\rho e)x \Vdash \phi(x)$$
Therefore, if \eqref{eq-ve} holds in the meta-theory then Set Induction applies, and for all $e$ witnessing the antecedent of \eqref{eq-ve} and all sets $x$, $(\rho e)x \Vdash \phi(x)$. This means $\rho$ witnesses \eqref{eq-ve} for any $\phi$. So, the four assumptions all pass through from the meta-theory. \end{proof}

\begin{mthm}\label{mthm-c-sc}\formalsys{CZF+RDC+ESP} can be interpreted in \formalsys{CZF} without Subset Collection (\formalsys{CZF^-}) in a way which preserves almost-negative formulas of arithmetic. That is, there is a translation of propositions $\phi \mapsto \phi^*$ such that, if \formalsys{CZF+RDC+ESP} proves $\phi$, then \formalsys{CZF^-} proves $\phi^*$; furthermore, if $\phi$ is an almost-negative formula of arithmetic, \formalsys{CZF-} proves $\phi \leftrightarrow \phi^*$. \end{mthm}
\begin{proof} Lemmas \ref{lem-c-xi} and \ref{lem-c-set} made use only of arguments available in \formalsys{CZF^-}, including the proof of the Class Inductive Definition Theorem \cite[\S 5]{AR} which specifically states it does not require Subset Collection. Therefore the four assumptions can be embedded into \formalsys{CZF^-}. The realisability interpretation of Meta-theorem \ref{mthm-c-ect} preserves almost-negative formulas of arithmetic, and the interpretation of set theory on top of that makes \formalsys{CZF+RDC+ESP} true. \end{proof}
\begin{rem} Relative consistency for \formalsys{ESP} was already known \cite[\S 8]{MR}, as it was for Subset Collection and \formalsys{RDC} \cite[\S 4]{RG}. This combined result is obtained directly without detours through Martin-L\"of type theory, Kripke-Platek set theory, or a classical theory of inductive definitions. \end{rem}

\section{The classical world through the looking glass}

Meta-theorem \ref{mthm-c-ect} also allows $\mathcal{V}$, $\mathcal{M}$, and $\mathcal{T}$ to be interpreted via realisability into a classical system, such as \formalsys{ID_1}, which is classical first-order arithmetic plus axioms for any non-nested positive inductive definitions. In this section, Truth and Falsehood are defined separately: \begin{itemize}
\item $\allnat{f,g} (\allnat{n} \defd{fn} \to \mathcal{F}(fn) \vee (\mathcal{T}(fn) \wedge \defd{gn} \wedge \mathcal{T}(gn))) \to \mathcal{T}(\Xi(f,g))$
\item $\allnat{f,g} (\allnat{n} \defd{fn} \to \mathcal{F}(fn) \vee (\mathcal{T}(fn) \wedge (\defd{gn} \to (\mathcal{T}(gn) \vee \mathcal{F}(gn))))) \wedge (\somenat{n} \defd{fn} \wedge \mathcal{T}(fn) \wedge (\defd{gn} \to \mathcal{F}(gn))) \to \mathcal{F}(\Xi(f,g))$
\item $\allnat{x} (\allnat{n} \defd{xn} \wedge (\mathcal{F}(\cd{x}{n}) \vee (\mathcal{T}(\cd{x}{n}) \wedge \mathcal{V}(\el{x}{n}))) \to \mathcal{V}(x)$
\end{itemize}
The corresponding induction principles are also assumed. These are positive mutually inductive definitions. They can be combined into a single positive non-nested inductive definition, and this is available in \formalsys{ID_1}. Then, $\mathcal{M}$ is an ordinary definition on top of this:
$$\mathcal{M}(x) \iff \mathcal{T}(x) \vee \mathcal{F}(x)$$

\begin{lem}Truth and Falsehood are mutually exclusive.\end{lem}
\begin{proof} In a manner similar to Lemma \ref{lem-c-con}, define the coherent versions of Truth and Falsehood: \begin{align*}
\mathcal{T}'(x) \iff& \mathcal{T}(x) \wedge \neg \mathcal{F}(x) &
\mathcal{F}'(x) \iff& \mathcal{F}(x) \wedge \neg \mathcal{T}(x) 
\end{align*}
Given $f$ and $g$, first suppose, as in the definition for $\mathcal{T}$, that:
$$\allnat{n} \defd{fn} \to \mathcal{F}'(fn) \vee (\mathcal{T}'(fn) \wedge \defd{gn} \wedge \mathcal{T}'(gn))$$
Immediately it follows that $\mathcal{T}(\Xi(f,g))$, but it also follows that $\neg\mathcal{F}(\Xi(f,g))$, because
$$\neg\somenat{n} \defd{fn} \wedge \mathcal{T}(fn) \wedge (\defd{gn} \to \mathcal{F}(gn))$$
So, $\mathcal{T}'(\Xi(f,g))$. Conversely suppose, as in the definition for $\mathcal{F}$, that: \begin{align*}
(\allnat{n} &\defd{fn} \to \mathcal{F}'(fn) \vee (\mathcal{T}'(fn) \wedge (\defd{gn} \to (\mathcal{T}'(gn) \vee \mathcal{F}'(gn)))))
\\ \wedge(\somenat{n} &\defd{fn} \wedge \mathcal{T}'(fn) \wedge (\defd{gn} \to \mathcal{F}'(gn))) \to \mathcal{F}'(\Xi(f,g))
\end{align*}
Immediately it follows that $\mathcal{F}(\Xi(f,g))$, but it also follows that $\neg\mathcal{T}(\Xi(f,g))$, because the second line implies that:
$$\neg\allnat{n} \defd{fn} \to \mathcal{F}(fn) \vee (\mathcal{T}(fn) \wedge \defd{gn} \wedge \mathcal{T}(gn))$$
So, $\mathcal{F}'(\Xi(f,g))$. This means that $\mathcal{T}'$ and $\mathcal{F}'$ satisfy the same mutual closure conditions as $\mathcal{T}$ and $\mathcal{F}$. By mutual $\mathcal{T}$ and $\mathcal{F}$ induction, it follows that:
$$\allnat{x} (\mathcal{T}(x) \to \mathcal{T}'(x)) \wedge (\mathcal{F}(x) \to \mathcal{F'}(x))$$
So finally,
$$\neg\somenat{x} \mathcal{T}(x) \wedge \mathcal{F}(x)$$
\end{proof} 

\begin{lem}In this interpretation \eqref{eq-vi}, \eqref{eq-ve}, \eqref{eq-xit}, and \eqref{eq-xim} are valid.\end{lem}
\begin{proof} Given any numbers $f$ and $g$, suppose for all $n$ such that $fn$ is defined, it is True or False. Suppose also that for all $n$ such that $fn$ and $gn$ are defined, and $fn$ is True, then $gn$ is True or False. This is where \formalsys{EM} comes in: either for all $n$ such that $fn$ is defined and True, $gn$ is also defined and True. In this case, $\Xi(f,g)$ is True. Or, there is some $n$ such that $fn$ is defined and True, but $gn$ is undefined, or defined but not True. But if $gn$ is defined, it is True or False. So, $gn$ is undefined or False. In that case, $\Xi(f,g)$ is False. So, $\Xi(f,g)$ is Meaningful, and \eqref{eq-xim} holds.

Suppose for all $n$ such that $fn$ is defined and True, $gn$ is also defined and True. Suppose that $\Xi(f,g)$ is False. Then there is an $n$ such that $fn$ is defined and True, and $gn$ is undefined or False. But $gn$ is defined an True, so, there would be a $gn$ which is both True and False. This is impossible by the previous lemma. So $\Xi(f,g)$ cannot be False; if it is Meaningful, then it is True. Conversely if it is True, then of course it is Meaningful. Also, for all $n$ such that $fn$ is defined, either $fn$ is False, or $fn$ is True and $gn$ is also defined and True. Given an $n$ such that $fn$ is defined and True, by the previous lemma it cannot also be False, so $gn$ is defined and True. So \eqref{eq-xit} holds.

Finally, to see that \eqref{eq-vi} holds, suppose for all $n$, $xn$ is defined, and $\cd{x}{n}$ is Meaningful, and if $\cd{x}{n}$ is True, then $\el{x}{n}$ is a Valid set. Since by the previous lemma $\cd{x}{n}$ cannot be True and False, this is equivalent to saying that $xn$ is defined, and either $\cd{x}{n}$ is False, or $\cd{x}{n}$ is True and $\el{x}{n}$ is Valid. That matches the above interpretation of $\mathcal{V}$. And \eqref{eq-ve} is the corresponding induction principle. \end{proof}

\begin{mthm}\label{mthm-c-id1} \formalsys{CZF^-} and \formalsys{ID_1} can interpret each other in a way that preserves $\Pi_2$ sentences of arithmetic.\end{mthm}
\begin{proof} The above shows that there is an interpretation in \formalsys{ID_1} that makes the four assumptions true, and then the realisability interpretation of Meta-theorem \ref{mthm-c-ect} preserves almost-negative sentences of arithmetic, which includes $\Pi_2$ sentences. In that interpretation \formalsys{CZF^-} (and more) is true. Conversely, \formalsys{ID_1} can be interpreted into \formalsys{ID_1(O)^i} (which is \formalsys{HA} plus a single non-nested inductive definition for the constructive second number class) in a way which preserves $\Pi_2$ sentences of arithmetic \cite{BFPS}. \formalsys{ID_1(O)^i} can then be embedded into \formalsys{CZF^-} using the same Class Inductive Definition Theorem used in the proof of Meta-theorem \ref{mthm-c-sc}. \end{proof}
\begin{rem} It was already well-known that \formalsys{CZF} and \formalsys{ID_1} had the same proof-theoretic strength \cite[\S 4]{RG}. This provides a simple interpretation without detours through type-theory or Kripke-Platek set theory. It can not, however, extend to almost-negative formulas as it did in Meta-theorem \ref{mthm-c-sc}. Markov's principle is almost-negative, and it is not provable in \formalsys{CZF}, whereas it is a classical tautology therefore provable in \formalsys{ID_1}. Indeed this technique does not work to interpret the four assumptions directly into \formalsys{ID_1^i}. In a way, it comes down to $\p1$: in \formalsys{ID_1} it can be divided in two, and in \formalsys{CZF} it is a class, but in \formalsys{ID_1^i} it is not expressible at all, other than by indirect interpretation of \formalsys{ID_1}.\end{rem}

\section{The equational interpretation}

Definition \ref{def-e-set} is a type of simultaneous inductive-recursive definition, as alluded to in section 1, although in fact it does not directly fit the normal inductive-recursive schema. It may be possible to come up with a general theory of inductive-multiple-recursive definitions. Instead this will be forced into the normal schema by also simultaneously defining {\it pairs} of sets. For clarity, the term ``w-set'' (and ``w-set-pair'') will be used in this context. 

\begin{defn}\label{def-e-set-pair}$\;$\begin{itemize}
\item $\empty$ is a w-set. Nothing is an intensional member of it.
\item If $f,g,h$ are (computable) sequences such that for all $n$, $\tuple{gn,hn}$ is a w-set-pair, and $fn$ is a w-set whenever $\tuple{gn,hn}$ is diagonal, then $[f;g,h]$ is a w-set. $x$ is an intensional member of it if and only if there is an $n$ such that $\tuple{gn,hn}$ is diagonal and $fn=x$. 
\item If $x$ and $y$ are w-sets, and if $\tuple{a,b}$ is a w-set-pair for every intensional member $a$ of $x$ and $b$ of $y$, then $\tuple{x,y}$ is a w-set-pair. It is diagonal if and only if for every intensional member $a$ of $x$ there is an intensional $b$ of $y$ such that $\tuple{a,b}$ is diagonal, and, for every intensional $b$ of $y$ there is an intensional member $a$ of $x$ such that $\tuple{a,b}$ is diagonal.
\item All w-sets and w-set-pairs are inductively generated by these rules.
\end{itemize}\end{defn}
$[f;g,h]$ represents the set $\setbn{n}{fn}{gn \approx hn}$, and $\empty$ is of course the empty set. This then has the form of a simultaneous inductive-recursive definition \cite{PD}. 

The introduction clause for w-set-pairs has two requirements: it requires that both elements are w-sets, but also that the intensional members of each of these are w-set-pairs. This is necessary to make the recursive definition of diagonality valid, since it must be able to refer to the diagonality of intensional members from each w-set, which requires these to be previously formed w-set-pairs. Of course, after the definition is set up, it can be shown by double w-set-induction that the second clause is redundant. Therefore w-set-pairs collapse to pairs of w-sets, and diagonality is extensional equality. All this is just to confirm the sense that Definition \ref{def-e-set} is not only constructively and predicatively valid, but actually a normal inductive-recursive definition with a bit of plastic surgery. There is no longer any need to speak of Definition \ref{def-e-set-pair} or w-set-pairs. 

The principles of set theory could now be re-proven in terms of w-sets instead of v-sets. This will only be sketched here. Rather, in this section, an isomorphism between w-sets and v-sets is exhibited. It turns out the sentences built using the $\Xi$ operator are exactly what is required. But in the v-set context, no limiting assumption on meaning was made. To make this work, Definition \ref{def-g-xi} needs to be extended by a clause that ``all Meaningful conditions are inductively built from $\Xi$.'' Formally, this will be an induction principle corresponding to \eqref{eq-xim}. Assume:
 \begin{align}\begin{split}
 (\allnat{f,g} &(\allnat{n} \defd{fn} \to \phi(fn)) \wedge (\allnat{n} \defd{fn} \wedge \defd{gn} \wedge \mathcal{T}(fn) \to \phi(gn)) \to \phi(\Xi(f,g))\\
    & ) \to \allnat{p} \mathcal{M}(p) \to \phi(p)
 \end{split}\label{eq-xime}
 \end{align}
This is a schema in $\phi$. It is not hard to see that the interpretation of section 7 already validates this induction principle since $\mathcal{M}$ was constructed via the Class Inductive Definition theorem. A similar argument applies in section 8.

Due to its limiting nature, \eqref{eq-xime} will inhibit a proof \formalsys{REA} such as the proof in the next section, and may actually allow it to be refuted. It can be considered unfaithful to the informal reading of ``meaningful'' in an essential way. This restriction, and therefore Definition \ref{def-e-set}, might be seen as undesirable, in much the same spirit that is often taken in classical set theory \cite[\S II.2]{PM}. On the other hand, if one is not interested in large set axioms (and most constructive mathematics can be formalised with much less than \formalsys{CZF+RDC+ESP}) then it may be seen as a desirable and natural completion that leaves absolutely no ambiguity as to what a set is.

Formal assumptions will now be stated. Introduce predicates $\mathcal{W}$ and $\mathcal{S}$. $\mathcal{W}(x)$ should be read ``$x$ is a w-set''. $\mathcal{S}(e,x,y)$ should be read ``$e$ witnesses the extensional equality of $x$ and $y$ as w-sets''.
Extensional equality of w-sets is defined as:
$$x \approx y \leftrightarrow \somenat{e} \mathcal{S}(e,x,y)$$
It is denoted $\approx$, to distinguish it from $\simeq$, extensional equality of v-sets, and $=$, intensional equality of natural numbers. Some numerical coding for $\empty$ and $[f;g,h]$ is assumed. Then, the formal defining assumptions for $\mathcal{W}$ are: \begin{equation}
\label{eq-wi1} \mathcal{W}(\empty) \end{equation}\begin{gather}
\label{eq-wi2} \begin{split} \allnat{f,g,h} (\allnat{n} \defd{fn} \wedge \defd{gn} \wedge \defd{hn} \wedge \mathcal{W}(gn) \wedge \mathcal{W}(hn) \wedge (gn \approx hn \to \mathcal{W}(fn))) 
 \\ \to \mathcal{W}([f;g,h]) \end{split} \\ 
\label{eq-we}\begin{split}
\phi(\empty) \wedge (\allnat{f,g,h} (\allnat{n} \defd{fn} \wedge \defd{gn} \wedge \defd{hn} \wedge \phi(gn) \wedge \phi(hn) \wedge (gn \approx hn \to \phi(fn))) \\ \; \to \phi([f;g,h])) \to \allnat{x} \mathcal{W}(x) \to \phi(x) 
\end{split}\end{gather}
\eqref{eq-we} is a schema in $\phi$. These three correspond to the first three clauses of Definition \ref{def-e-set}, in order. The last clause will be expressed by axioms for $\mathcal{S}$, like \eqref{eq-rf}, but here this has to be done by cases. The first three cases are not too hard:
\begin{align}
\label{eq-sf1}\allnat{e} &\mathcal{S}(e,\empty,\empty) \\
\label{eq-sf2}\allnat{e,f,g,h} &\mathcal{W}([f;g,h]) \to (\mathcal{S}(e,\empty,[f;g,h]) \leftrightarrow \neg \somenat{n} gn \approx hn) \\
\label{eq-sf3}\allnat{e,f,g,h} &\mathcal{W}([f;g,h]) \to (\mathcal{S}(e,[f;g,h],\empty) \leftrightarrow \neg \somenat{n} gn \approx hn)
\end{align}
Finally the main case is straightforward, though verbose:
\begin{equation}\label{eq-sf4}\begin{matrix}
\allnat{e,f,g,h,j,k,l} \mathcal{W}([f;g,h]) \wedge \mathcal{W}([j;k,l]) \to (\mathcal{S}(e,[f;g,h],[j;k,l]) \leftrightarrow \hfill \\
\begin{matrix}
   \hfill(\allnat{x} & \mathcal{S}(x_R,g(x_L),h(x_L)) & \to \defd{e_Lx} \hfill \\
   && \wedge \;\mathcal{S}((e_Lx)_{RL}, k((e_Lx)_L), l((e_Lx)_L)) \hfill \\
   && \wedge \;\mathcal{S}((e_Lx)_{RR},f(x_L),j((e_Lx)_L))) \hfill \\
   \hfill \wedge  (\allnat{x} & \mathcal{S}(x_R,k(x_L),l(x_L)) & \to \defd{e_Rx} \hfill \\
   && \wedge \;\mathcal{S}((e_Rx)_{RL}, g((e_Rx)_L), h((e_Rx)_L)) \hfill \\
   && \wedge \;\mathcal{S}((e_Rx)_{RR},j(x_L),f((e_Rx)_L))) ) \hfill
\end{matrix}\end{matrix}\end{equation}
The assumption \formalsys{ECT_W} is the same schema as \formalsys{ECT_0}, but where the class of almost-negative sentences includes the predicates $\mathcal{W}$ and $\mathcal{S}$. The predicates $\mathcal{W}$ and $\mathcal{S}$ and their defining assumptions will be used for this section only. 

Similar to section 2, $IND$ and $EL$ can be defined. The proofs of sections 2 and 3 go through basically as is. The operator $\Pi$ can be interpreted as follows:
$$\Pi(x \approx y, u \approx v) \iff \stuple{x,\setb{u}{x \approx y}} \approx \stuple{y,\setb{v}{x \approx y}}$$
There is a difficulty when it comes to \formalsys{ECT_W} that does not occur with \formalsys{ECT_V}: the predicate $IND$ is not almost-negative, because a w-set equality is not almost-negative. To get around this, an equation $x \approx y$ is replaced by $\mathcal{S}(e,x,y)$ with an extra parameter $e$. Every $\mathcal{S}(e,x,y)$ can be expressed as a different equation $u \approx v$, and this can be shown in the manner of Lemma \ref{lem-e-mean} below, using Set Induction rather than \eqref{eq-xime}. This allows w-sets to be constructed using $\mathcal{S}$'s as conditions. This in turn allows the proofs of section 4 to go through, with the necessary modifications. Finally, sections 5 and 6 are easy: the key was Kronecker Delta, Lemma \ref{lem-g-kron}, and this is now trivial.

The above work-around is also the basic idea required to get an isomorphism, and this is formally exhibited here. An equation $x \approx y$ is called ``canonically witnessed'' if there is a number $f$ such that $x \approx y \to \mathcal{S}(f,x,y)$. 

\begin{lem}\label{lem-e-impl}Suppose that $x$ and $y$ are w-sets and that $x \approx y$ is canonically witnessed. Suppose also that $x \approx y$ implies that $u$ and $v$ are w-sets and that $u \approx v$ is canonically witnessed. Then there are w-sets $s$ and $t$ such that $s \approx t$ is canonically witnessed, and such that $s \approx t \leftrightarrow (x \approx y \to u \approx v)$.\end{lem}
\begin{proof} Define: \begin{align*}
s &= \setbn{\zeta}{u}{x \approx y} \\
t &= \setbn{\zeta}{v}{x \approx y}
\end{align*}
$\zeta$ is a dummy variable. $s$ and $t$ are valid w-sets under the stated assumptions. Clearly, $s \approx t \leftrightarrow (x \approx y \to u \approx v)$. All the relevant sub-equations are canonically witnessed, so a canonical witness for $s \approx t$ can be constructed. \end{proof}
 
\begin{lem}\label{lem-e-seq}Given a sequence of canonically witnessed w-set equations, there is a canonically witnessed w-set equation which expresses that every equation in the sequence is true.\end{lem}
 \begin{proof} Given sequences of w-sets $x$ and $y$ such that $xn \approx yn$ is canonically witnessed for all $n$, define:\begin{align*}
 u &= \setbn{n}{\stuple{\neumarral{n},xn}}{\neumarral{0} \approx \neumarral{0}} \\
 v &= \setbn{n}{\stuple{\neumarral{n},yn}}{\neumarral{0} \approx \neumarral{0}}
 \end{align*}
Now, by the properties of set-theoretic ordered pair and finite Von Neumann ordinals, the equation $\stuple{\neumarral{n},xn} \approx \stuple{\neumarral{n'},yn'}$ is true if and only if $n=n'$ and $xn \approx yn'$. It follows that $u \approx v$ if and only if $xn \approx yn$ for all $n$. Moreover, an equation between any two finite Von Neumann ordinals is canonically witnessed; this can easily be shown by induction. Also by hypothesis $xn \approx yn$ is canonically witnessed. Therefore a canonical witness for $u \approx v$ can be constructed.  \end{proof}
 
\begin{lem}\label{lem-e-def2}Given two numbers $e$ and $i$,  suppose that if $ei$ is defined, it is a canonically witnessed w-set equation, that is, $(ei)_{RL}$ and $(ei)_{RR}$ are w-sets and $(ei)_L$ is a canonical witness for $(ei)_{RL} \approx (ei)_{RR}$. Then there is a canonically witnessed w-set equation which expresses that $ei$ is defined and true, that is, it expresses $\defd{ei} \wedge (ei)_{RL} \approx (ei)_{RR}$.\end{lem}
\begin{proof} Define:\begin{align*}
x &= \neumarral{1}\\
y &= \setbn{u}{\neumarral{0}}{T'(e,i,u,\mathbf{U}(u)_{RL},\neumarral{0}) \approx T'(e,i,u,\mathbf{U}(u)_{RR},\neumarral{1}} \\
T'(e,i,u,a,b) &= \begin{cases} 
 a, &\text{if } \mathbf{T}(e,i,u) \\ 
 b, &\text{otherwise} 
\end{cases}\end{align*}
The equation in the definition of $y$ is true if and only if $\mathbf{T}(e,i,u)$ and $\mathbf{U}(u)_{RL} \approx \mathbf{U}(u)_{RR}$. So, the equation $x \approx y$ is true if and only if $\defd{ei} \wedge (ei)_{RL} \approx (ei)_{RR}$, as required. Now, to get a canonical witness $f$, notice that all the relevant sub-equations are canonically witnessed. $f$ just has to produce a $u$ such that $\mathbf{T}(e,i,u)$. It can do an unbounded search. If in fact $x \approx y$, then $\defd{ei}$, so the search will terminate; $x \approx y \to \mathcal{S}(f,x,y)$, as required. \end{proof}
 
\begin{lem}\label{lem-e-mean}Any Meaningful condition can be expressed by a canonically witnessed w-set equation.\end{lem}
\begin{proof} Proceed by induction using \eqref{eq-xime}. Given $f$ and $g$, take for inductive hypothesis that when $fn$ is defined it can be expressed by a canonically witnessed equation, and that when $fn$ and $gn$ are defined and $fn$ is True, $gn$ can be expressed by a canonically witnessed equation. 
 
Take any $n$. By Lemma \ref{lem-e-def2}, there is a canonically witnessed equation that expresses that $fn$ is defined and True; and also, if $fn$ is defined and True, there is a canonically witnessed equation that expresses that $gn$ is defined and True. So, by Lemma \ref{lem-e-impl}, there is a canonically witnessed equation that expresses that if $fn$ is defined and True, then $gn$ is defined and True. 

So, by Lemma \ref{lem-e-seq}, there is a canonically witnessed equation expressing $\Xi(f,g)$. And finally, by \eqref{eq-xime}, every Meaningful condition can be expressed by a canonically witnessed equation. \end{proof}
 
\begin{lem}\label{lem-e-set}If $g$ and $h$ are sequences of v-sets, and if $f$ is a sequence such that $fn$ is a v-set if $gn \simeq hn$, then $\setbn{n}{fn}{gn \simeq hn}$ is a v-set.\end{lem}
\begin{proof} It follows from Lemmas \ref{lem-g-req} and \ref{lem-g-eq} that
$$\setbn{n}{fn}{gn \simeq hn} = \Lambda n.\tuple{f(n_L),\overline{\mathcal{R}}(n_R,g(n_L),h(n_L))}$$
is the required v-set. \end{proof}
 
\begin{mthm}\label{mthm-iso} Definition \ref{def-e-set} is isomorphic to Definitions \ref{def-g-set} and \ref{def-g-xi} plus the assumption that all Meaningful conditions are inductively built using $\Xi$. Specifically, \formalsys{HA+ECT_V} plus the assumptions \eqref{eq-vi}, \eqref{eq-ve}, \eqref{eq-xit}, \eqref{eq-xim}, and \eqref{eq-xime} can interpret \formalsys{HA+ECT_W} plus the assumptions \eqref{eq-wi1}--\eqref{eq-sf4}, and vice versa, in a way that preserves all sentences of first-order set theory.\end{mthm}
 \begin{proof} Given a w-set $x$, an equivalent v-set $\mathbf{W}x$ can be constructed recursively using Empty Set and Lemma \ref{lem-e-set}. Then interpret:\begin{align*}
 \mathcal{W}(x) &\iff \mathcal{V}(\mathbf{W}x) \\
 \mathcal{S}(e,x,y) &\iff \mathcal{R}(e,\mathbf{W}x,\mathbf{W}y) 
 \end{align*}
 $\mathcal{R}$ can in turn be interpreted in terms of $\mathcal{T}$ by Lemma \ref{lem-g-req}. This will satisfy the assumptions for w-sets. Also since $\mathcal{V}$ and $\mathcal{T}$ are almost-negative, so are $\mathcal{W}$ and $\mathcal{S}$, therefore \formalsys{ECT_W} is also satisfied.
 
The other direction works as follows: given a condition $p$, Lemma \ref{lem-e-mean} gives a triple $\mathbf{M}p = \tuple{e,\tuple{x,y}}$. Given a v-set $x$, an equivalent w-set $\mathbf{V}x$ can be constructed by recursively replacing every condition $p$ which appears with an equation $(\mathbf{M}p)_{RL} \approx (\mathbf{M}p)_{RR}$. Then interpret:\begin{align*}
 \mathcal{M}(p) &\iff \mathcal{W}((\mathbf{M}p)_{RL}) \wedge \mathcal{W}((\mathbf{M}p)_{RR}) \\
 \mathcal{T}(p) &\iff \mathcal{S}((\mathbf{M}p)_L, (\mathbf{M}p)_{RL}, (\mathbf{M}p)_{RR}) \\
 \mathcal{V}(x) &\iff \mathcal{W}(\mathbf{V}x)
 \end{align*}
Since $(\mathbf{M}p)_{RL} \approx (\mathbf{M}p)_{RR}$ is canonically witnessed by $(\mathbf{M}p)_L$, the equation is interchangeable with $\mathcal{T}(p)$, but $\mathcal{T}$ remains almost-negative. So, this satisfies the assumptions for v-sets, and $\mathcal{V}$, $\mathcal{M}$ and $\mathcal{T}$ are almost-negative because $\mathcal{W}$ and $\mathcal{S}$ are, so \formalsys{ECT_V} is also satisfied. 
 
It is not hard to see that the operations $\mathbf{V}$ and $\mathbf{W}$ are also inverses up to extensional equality, so first-order set theory is preserved. \end{proof}
 
The formal machinery of w-sets might seem inelegant compared to that of v-sets. It might then be imagined that if an actual construction were attempted in terms of w-sets equations, the result would be even messier. This section closes with a brief example of a direct construction.

Assume bijective encodings $\mathbf{Q}$ and $\mathbf{Q^+}$ of the rational numbers and the positive rational numbers, respectively. Define: \begin{align*}
J(e,z,u,x) &= \begin{cases}
  \neumarral{1} & \text{if } \mathbf{T}(e,z,u) \wedge \mathbf{Q}x < \mathbf{Q}\mathbf{U}(u) - \mathbf{Q^+}z \\ 
 \neumarral{0} & \text{otherwise} 
\end{cases}\\
K(e,z,z',u,u') &= \begin{cases}
   \neumarral{1} & \text{if } \mathbf{T}(e,z,u) \wedge \mathbf{T}(e,z',u') \wedge \abs{\mathbf{Q}\mathbf{U}(u) - \mathbf{Q}\mathbf{U}(u')} < \max(\mathbf{Q^+}z,\mathbf{Q^+}z') \\
  \neumarral{0} & \text{otherwise}
\end{cases}
 \end{align*}
The rational inequations involved are of course computable, so $J$ and $K$ are computable. Now define: \begin{align*}
X &= \Lambda e. \setbn{x,z,u}{\mathbf{Q}x}{\neumarral{1} \approx J(e,z,u,x)}\\
Y &= \Lambda e. \setbn{z,z',u,u'}{\neumarral{\tuple{z,z'}}}{\neumarral{1} \approx K(e,z,z',u,u')}\\
\mathbb{R} &= \setbn{e}{Xe}{\omega \approx Ye}
\end{align*}
Viewing $e$ as a computable partial function from $\mathbb{Q}^+$ to $\mathbb{Q}$, it follows that 
$$Xe \approx \setb{x\in\mathbb{Q}}{\somein{z}{\mathbb{Q}^+} \defd{ez} \wedge x < ez - z}$$
The equation $\omega \approx Ye$ expresses that
$$(\allin{z}{\mathbb{Q}^+} \defd{ez}) \wedge (\allin{z,z'}{\mathbb{Q}^+} \left| ez - ez' \right| < \max(z,z'))$$
When this equation holds, it follows that $e$ is a Cauchy sequence, and that $Xe$ is the corresponding located lower cut of rationals. It can be shown in \formalsys{CZF} that to any located lower cut there is such a Cauchy sequence and vice versa \cite[\S 3.6]{AR}. This does rely on Countable Choice, which is of course valid here. So, although intensionally $\mathbb{R}$ was constructed in terms of Cauchy sequences, extensionally it is the w-set of all located Dedekind reals.

\section{Regular sets}

The interpretation will transcend \formalsys{CZF} if Meaning is extended to include conditions that cannot be expressed with $\Xi$ alone. For instance, a basic reflection principle should be sufficient to prove the consistency of \formalsys{CZF}. But there is a more natural assumption from which \formalsys{REA} will follow.

The concept of a {\it bar} plays a significant role in Brouwer's school of intuitionism. A bar is a subset of finite sequences of natural numbers, such that every infinite sequence has at least one of its finite prefixes belonging to it. Here a slight liberty is taken with the definition: a bar will be a such that no infinite sequence has all of its prefixes belonging to it. Also, no form of bar induction will be assumed. Instead the bars that are considered are those that are inductively generated in the first place.

Assume some elementary bijection between natural numbers and finite sequences of natural numbers, which will also be written with angle brackets. No confusion should result. The $\star$ operator will denote sequence concatenation. Also $\abs{x}$ will denote the length of the sequence $x$.

\begin{defn}\label{def-o-bar}$\;$\begin{itemize}
\item If $f\tuple{}$ being True would imply that $\Lambda x.f(\tuple{n} \star x)$ are inductively generated bars for every $n$, then, $f$ is an inductively generated bar.
\item All inductively generated bars are inductively generated by this rule.
\end{itemize}\end{defn}
It follows from this definition that if $f\tuple{} = \bot$, then $f$ is an inductively generated bar; and, if $g$ is a sequence of inductively generated bars, then $\sup g$ is an inductively generated bar, where $(\sup g)\tuple{} = \top$ and $(\sup g)(\tuple{n} \star x) = gnx$. But the definition allows for conditions other than $\top$ and $\bot$.
\begin{defn}\label{def-o-omega}$\Omega(f)$ is a condition which expresses that $f$ is an inductively generated bar. It is Meaningful if $fn$ is defined and Meaningful for all $n$. \end{defn}

Formal assumptions on $\mathcal{M}$ and $\mathcal{T}$ are made following Definitions \ref{def-o-bar} and \ref{def-o-omega}. In section 5 the odd numbers were used to denote applications of $\Xi$, and here the numerical convention that $\Omega(f) = 4 \cdot f + 2$ is adopted, leaving unspecified the Meaning of conditions whose numbers are divisible by four. Assume:\begin{align}
\allnat{f} &(\allnat{n} \defd{fn} \wedge \mathcal{M}(fn)) \to \mathcal{M}(\Omega(f)) \label{eq-om}\\
\allnat{f} &\mathcal{M}(\Omega(f)) \wedge (\mathcal{T}(f\tuple{}) \to \allnat{n} \mathcal{T}(\Omega(f^{(n)}))) \to \mathcal{T}(\Omega(f)) \label{eq-oti}\\
(\allnat{f} &\mathcal{M}(\Omega(f)) \wedge (\mathcal{T}(f\tuple{}) \to \allnat{n} \phi(f^{(n)})) \to \phi(f)) \to (\allnat{f} \mathcal{T}(\Omega(f)) \to \phi(f)) \label{eq-ote}
\end{align}
\eqref{eq-ote} is a schema in $\phi$. The expression $\Lambda x.f(\tuple{n} \star x)$ has been abbreviated $f^{(n)}$. Truth for $\Omega$ is given by a standard accessibility definition which is immediately constructively valid.

\begin{lem}\label{lem-o-pi-bar}Suppose $p$ is Meaningful, $x$ is a sequence, and $p$ being True would imply that $x$ is an inductively generated bar. Then, $\Pi'(p,x)$ is an inductively generated bar, where $\Pi'(p,x) = \Lambda m. \Pi(p,xm)$.\end{lem}
\begin{proof} $p$ being True would imply that $x$ is an inductively generated bar. So it would also imply that $x$ is a sequence of Meaningful conditions. That is, it would imply that $xm$ is Meaningful for all $m$, so, by Definition \ref{def-g-pi}, $\Pi(p,xm)$ is (unconditionally) Meaningful for all $m$. Therefore $\Pi'(p,x)$ is a sequence of Meaningful conditions, and it is Meaningful to ask if it is an inductively generated bar.

$p$ being True would also imply that $\Pi(p,xm)$ is True if and only if $xm$ is True. So, $p$ being True would imply that $\Pi'(p,x)$ is equivalent to $x$, and so is also an inductively generated bar. 

Now suppose $\Pi'(p,x)\tuple{}$ is True. It would follow that $p$ is True, so $\Pi'(p,x)$ would be an inductively generated bar, and so $\Pi'(p,x)^{(n)}$ would be an inductively generated bar for all $n$. Since $\Pi'(p,x)\tuple{}$ being True would imply that $\Pi'(p,x)^{(n)}$ is an inductively generated bar for all $n$, it follows that $\Pi'(p,x)$ is (unconditionally) an inductively generated bar. \end{proof}

\begin{lem}Suppose $\tau$ is a sequence of Meaningful conditions. Then there are $B(\tau,x)$ such that:\begin{itemize}
\item $B(\tau,x)$ is Meaningful for all $x$.
\item $B(\tau,x)$ is True if $B(\tau,\el{x}{n})$ is True for all $n$ such that $xn$ is defined and such that $\tau(\cd{x}{n})$ is True.
\item All the $x$'s for which $B(\tau,x)$ is True are inductively generated by this rule.
\end{itemize}\end{lem}\begin{proof} Define the conditions $B(\tau,x)$ as follows:\begin{align*}
B(\tau,x) &= \Omega(\Lambda z.A(\tau,x,z)) \\
A(\tau,x,z) &= \top, \text{ if } \abs{z}<2 \\
A(\tau,x,\tuple{n,u} \star z) &= \Pi(\tau(\mathbf{U}(u)_R), A(\tau,\mathbf{U}(u)_L,z)), \text{ if } \mathbf{T}(x,n,u), \text{ else } \bot
\end{align*}
The conditions $A(\tau,x,z)$ are either equal to $\top$ or else is formed from $\Pi$ and $\tau n$ for some values of $n$. Since $\tau$ is a sequence of Meaningful conditions, it follows that $A(\tau,x,z)$ is always Meaningful. So by Definition \ref{def-o-omega}, $B(\tau,x)$ is Meaningful. 

Now, $A(\tau,x,z) = \top$ if $\abs{z}<2$, so by the definition, $\Lambda z.A(\tau,x,z)$ is an inductively generated bar if and only if $\Lambda z.A(\tau,x,\tuple{n,u} \star z)$ is an inductively generated bar for all $n$ and $u$. Also, unless $\mathbf{T}(x,n,u)$, 
 $A(\tau,x,\tuple{n,u} \star z) = \bot$, so it is trivially an inductively generated bar, by the definition. If $\mathbf{T}(x,n,u)$ then  $A(\tau,x,\tuple{n,u} \star z) = \Pi(\tau(\mathbf{U}(u)_R), A(\tau,\mathbf{U}(u)_L,z))$. By the previous lemma, $\Lambda z.\Pi(\tau(\mathbf{U}(u)_R), A(\tau,\mathbf{U}(u)_L,z))$ is an inductively generated bar if and only if $\tau(\mathbf{U}(u)_R)$ being True would imply that $\Lambda z.A(\tau,\mathbf{U}(u)_L,z)$ is an inductively generated bar.

It follows, then, $\Lambda z.A(\tau,x,z)$ is an inductively generated bar if $\Lambda z.A(\tau,\el{x}{n},z)$ is an inductively generated bar for all $n$ such that $xn$ is defined and such that $\tau(\cd{x}{n})$ is True; and, this gives all the $x$'s for which $\Lambda z.A(\tau,x,z)$ is an inductively generated bar. Finally, $B(\tau,x)$ is True if and only if $\Lambda z.A(\tau,x,z)$ is an inductively generated bar, and this matches the last two clauses of the lemma. \end{proof}

\begin{lem}\label{lem-o-vdef} Suppose $\tau$ is a sequence of Meaningful conditions. Then there are conditions $W(\tau,x)$ such that:\begin{itemize}
\item $W(\tau,x)$ is Meaningful for all $x$.
\item $W(\tau,x)$ is True if $xn$ is defined for all $n$, and if $W(\tau,\el{x}{n})$ is True for all $n$ such that $\tau(\cd{x}{n})$ is True.
\item All the $x$'s for which $W(\tau,x)$ is True are inductively generated by this rule.
\end{itemize}\end{lem}\begin{proof}
Define the conditions $W(\tau,x)$ as follows:\begin{align*}
W(\tau,x) &= \Pi(B(\tau,x),L(\tau,x)) \\
L(\tau,x) &= \Pi(TOT(x),J(\tau,x)) \\
TOT(x) &= \Xi(\Lambda n.\top, \Lambda n.K(\top,xn)) \\
K(a,b) &= a \\
J(\tau,x) &= \Xi(\Lambda n.\tau(\cd{x}{n}), \Lambda n.L(\tau,\el{x}{n})) 
\end{align*}
Here, $B$ is given by the previous lemma. $TOT(x)$ asserts that $xn$ is always defined. $J(\tau,x)$ asserts that $L(\tau,\el{x}{n})$ is True for all $n$ such that $\tau(\cd{x}{n})$ is True. Therefore, $L(\tau,x)$ asserts that $xn$ is defined for all $n$, and that $L(\tau,\el{x}{n})$ is True for all $n$ such that $\tau(\cd{x}{n})$ is True. Provided that it is Meaningful, it has the correct Truth conditions.

$TOT(x)$ is always Meaningful. If $TOT(x)$ is True and $L(\tau,\el{x}{n})$ is Meaningful for all $n$ such that $\tau(\cd{x}{n}$ is True, then $J(\tau,x)$ is Meaningful. So, if $L(\tau,\el{x}{n})$ is Meaningful for all $n$ such that $\tau(\cd{x}{n})$ is defined and True, then $L(\tau,x)$ is Meaningful. 

The previous lemma establishes an induction principle which matches a recursive condition for $L(\tau,x)$ to be Meaningful. So, if $B(\tau,x)$ is True, then $L(\tau,x)$ is Meaningful. It follows that $W(\tau,x)$ is always Meaningful, as claimed. \end{proof}

\begin{lem}Suppose $\tau$ is a sequence of Meaningful conditions, and that there is a number $F$ such that $\tau F$ is not True. Then there is a v-set $V(\tau)$ containing all and only the v-sets that can be ``built from'' $\tau$. That is: \begin{itemize}
\item If all the conditions of $y$ are in $\tau$ (that is, for all $n$ there is an $m$ such that $\tau m$ is True if and only if $\cd{y}{n}$ is True), and all the elements of $y$ are in $V(\tau)$, then $y$ is in $V(\tau)$.
\item All the elements of $V(\tau)$ are inductively generated by this rule.
\end{itemize}
\end{lem}
\begin{proof} Define the v-sets $Y(\tau,x)$ and $V(\tau)$ as follows:\begin{align*}
Y(\tau,x) &= \Lambda n.\tuple{Y(\tau,\el{x}{n}),\tau(\cd{x}{n})} \\
V(\tau) &= \Lambda x.\tuple{Y(\tau,x),W(\tau,x)}
\end{align*}
Here $W$ is given by the previous lemma. Take for an inductive hypothesis that $xn$ is defined for all $n$, and that $Y(\tau,\el{x}{n})$ is a v-set for all $n$ such that $\tau(\cd{x}{n})$ is True. $Y(\tau,x)$ would be v-set. So, from the induction principle established by the previous lemma, it follows that $Y(\tau,x)$ is a v-set for all $x$ such that $W(\tau,x)$ is True. This establishes that $V(\tau)$ is a v-set.

Suppose $y$ is a v-set and can be built from $\tau$. Then for all $n$, there is an $m$ such that $\cd{y}{n}$ is True if and only if $\tau m$ is True. So by \formalsys{ECT_V}, there is an $e$ such that
$$\allnat{n} \defd{en} \wedge (\mathcal{T}(\cd{y}{n}) \leftrightarrow \mathcal{T}(\tau(en)))$$
Also since $y$ can be built from $\tau$, every element of $y$ can be built from $\tau$. Take, for a set-inductive hypothesis, that for every element $b$ of $y$ there is an $a$ such that $W(\tau,a)$ is True and $b \simeq Y(\tau,a)$. So again by \formalsys{ECT_V}, there is an $f$ such that
$$\allin{n}{IND(y)} \defd{fn} \wedge \mathcal{T}(W(\tau,fn)) \wedge \el{y}{n} \simeq Y(\tau,fn)$$
So, define:\begin{align*}
z &= \Lambda n.\tuple{\mathbf{U}(n_R)_L, T_0(x,n_L,n_R,\mathbf{U}(n_R)_R)} \\
T_0(x,n,u,p) &= p \text{ if } \mathbf{T}(x,n,u) \text{ else } F \\
x &= \Lambda n.\tuple{fn,en} 
\end{align*}
$en$ is always defined. If $\tau(en)$ is True then so is $\cd{y}{n}$, so $fn$ is defined and $W(\tau,fn)$ is True. The proof of Lemma \ref{lem-g-part} applies with minor changes to show that $zn$ is defined for all $n$, and that, if $\tau(\cd{z}{n})$ is True for some $n$, then $\el{z}{n} = fn'$ for some $n' \in IND(y)$. By the properties of $f$, $W(\tau,fn')$ is True for all $n' \in IND(y)$. So, by the previous lemma, $W(\tau,z)$ is True, and as established above, this means that $Y(\tau,z)$ is a v-set.

Conversely, for all $n' \in IND(y)$ there is an $n$ such that $\tau(\cd{z}{n})$ is True and $\el{z}{n} = fn'$, and from the properties of $f$ it follows that $\el{y}{n'} \simeq Y(\tau,\el{z}{n})$. So, $Y(\tau,z) \simeq y$. Finally, by Set Induction, it follows that for any v-set $y$ which can be built from $\tau$, there is a $z$ such that $W(\tau,z)$ is True and $Y(\tau,z) \simeq y$, which means that $V(\tau)$ contains $y$. $V(\tau)$ contains only such v-sets, because when $W(\tau,x)$ is True, $Y(\tau,x)$ is well-founded and directly built from $\tau$. \end{proof}

\begin{lem}\label{lem-o-reg} Suppose, as in the previous lemma, that $\tau$ is a sequence of Meaningful conditions, and that there is a number $F$ such that $\tau F$ is not True. Then $V(\tau)$, the v-set of all v-sets which can be built from $\tau$, is in fact ``regular'': inhabited, transitive, and a model of Strong Collection. That is, given an element $x$ of $V(\tau)$ and any v-set $R$ which is a binary relation such that to every element of $x$ there is an element of $V(\tau)$ related to it by $R$, then, there is an element $y$ of $V(\tau)$ which contains only v-sets related to elements of $x$, and which contains, for each element of $x$, at least one v-set related to it. \end{lem}
\begin{proof} $F$ is a condition in $\tau$ which is not True, so the empty v-set, at least, can be built from it, and therefore $V(\tau)$ is inhabited. Transitivity follows from the fact that if a v-set can be built from $\tau$, then by definition so can all its elements. The previous lemma shows that $V(\tau)$, $\Lambda n.\top$ and $\tau$ mirror the assumptions \eqref{eq-vi} and \eqref{eq-ve} on $\mathcal{V}$, $\mathcal{M}$, and $\mathcal{T}$, respectively. The proof of Theorem \ref{thm-g-coll} (including Lemma \ref{lem-g-part}) only required these assumptions plus the fact that $\bot$ is Meaningful and not True. $R$ can be used as the relation $\psi$, even though $R$ may be external to $V(\tau)$, because $\psi$ was allowed to be arbitrary. So, Strong Collection applies inside $V(\tau)$. \end{proof}

\begin{lem}\label{lem-o-mseq} For every v-set, there is a sequence of Meaningful conditions from which the v-set can be built.\end{lem}
\begin{proof} Define $h(x)$ as follows:\begin{align*}
h(x) &= \Lambda m. h'(x,m) \\
h'(x,m) &= \top \text{ if } \abs{m} < 2 \\
h'(x,\tuple{n,u} \star m) &= \Pi(\mathbf{U}(u)_R,h'(\mathbf{U}(u)_L,m)), \text{ if } \mathbf{T}(x,n,u), \text{ else } \bot \\
\end{align*}
Now, $h'(x,\tuple{n,u} \star m) = \Pi(\mathbf{U}(u)_R,h'(\mathbf{U}(u)_L,m))$ if $\mathbf{T}(x,n,u)$. $h'(\mathbf{U}(u)_L,\tuple{}) = \top$, so $h'(x,\tuple{n,u})$ is True if and only if $\mathbf{U}(u)_R$ is True and $\mathbf{T}(x,n,u)$. So $h(x)$ contains the conditions $\cd{x}{n}$ for all $n$ such that $xn$ is defined. If $\cd{x}{n}$ is True, then, $\Pi(\cd{x}{n},h'(\el{x}{n},m))$ is True if and only if $h'(\el{x}{n},m)$ is True. So, $h(x)$ also contains all the conditions in $h(\el{x}{n})$ for all $n$ such that $xn$ is defined and $\cd{x}{n}$ is True.  

Given a v-set $x$, take for a set-inductive hypothesis that for every element $y$ of $x$, $h(y)$ is a sequence of Meaningful conditions from which $y$ can be built. It follows that $h(x)$ is a sequence of Meaningful conditions. $h(x)$ contains all the conditions in $h(\el{x}{n})$ for all $n \in IND(x)$, so every element of $x$ can be built from $h(x)$. $h(x)$ also contains the conditions $\cd{x}{n}$. So, $x$ can be built from $h(x)$. 

By Set Induction, then, any v-set $x$ can be built from the sequence of Meaningful conditions $h(x)$. \end{proof}

\begin{thm}[Regular Extension] Every v-set is contained in a regular v-set.\end{thm}
\begin{proof} Given any v-set $x$, there is a sequence of Meaningful conditions from which $x$ can be built. This can be trivially extended to include $\bot$. There is a regular v-set containing all and only the v-sets which can be built from this sequence. By construction of this sequence, $x$ itself is such a v-set, so it is contained in this regular v-set. \end{proof}

\section{\formalsys{CZF+REA} through the looking glass}

The interpretation of section 7 can be extended to include the assumptions \eqref{eq-om}, \eqref{eq-oti}, and \eqref{eq-ote}. First an inductive class definition is given to capture the Truth conditions of $\Omega$. Define the classes: \begin{align*}
\Upsilon(M) &= \setb{\stuple{B,f}}{f \in W(M) \wedge B \simeq V(f,M)} \\
V(f,M) =& \setb{f^{(n)}}{n \in \omega \wedge \stuple{f\tuple{}, \neumarral{1}} \in M} \\
W(M) =& \setb{f \in \omega}{\allin{n}{\omega} \defd{fn} \wedge \someset{h} \stuple{fn,h} \in M}
\end{align*}
Form the least $\Upsilon(M)$-closed class, $B(M)^\infty$. The role of \formalsys{REA} will be to show that $B(M)^\infty$ is in fact a set when $M$ is. Then define the class:
$$\Phi' = \Phi \cup \setb{\stuple{M,a}}{\somein{f}{W(M)} a \simeq \stuple{\Omega(f),\setb{\neumarral{0}}{f \in B(M)^\infty}}}$$
$\Phi$ is the same operator as in section 7, but now $M^\infty$ is the least $\Phi'$-closed class. $V^\infty$ is defined the same way, and $\mathcal{V}$, $\mathcal{M}$, and $\mathcal{T}$ are interpreted the same way.

\begin{lem}\label{lem-cr-binf} Suppose that $M$ is a subset of $\omega \times \p1$. Then:\begin{itemize}
\item If $f \in W(M)$, and if $\stuple{f\tuple{},\neumarral{1}} \in M$ implies that $f^{(n)} \in B(M)^\infty$ for all $n \in \omega$, then, $f \in B(M)^\infty$.
\item All the elements of $B(M)^\infty$ are inductively generated by this rule.
\end{itemize}\end{lem}
\begin{proof} First, note that if $f \in W(M)$, then $f^{(n)} \in W(M)$ for all $n$. Now, suppose
$$\stuple{f\tuple{},\neumarral{1}} \in M \to \allin{n}{\omega} f^{(n)} \in B(M)^\infty$$
It follows that $V(f,M) \subseteq B(M)^\infty$, and so $f \in B_M^\infty$, since $B(M)^\infty$ is $\Upsilon(M)$-closed. Conversely, if there is a class $B'$ such that
$$\allin{f}{W(M)} (\stuple{f\tuple{},\neumarral{1}} \in M \to \allin{n}{\omega} f^{(n)} \in B') \to f \in B'$$
Then, $B'$ is $\Upsilon(M)$-closed, and $B(M)^\infty$ is the least such class so $B(M)^\infty \subseteq B'$, giving the required induction principle. \end{proof}

\begin{lem}\label{lem-cr-bar} If $M$ is a subset of $\omega \times \p1$, then $B(M)^\infty$ is a set. \end{lem}
\begin{proof} First, $V(f,M)$ is indeed a set, being given by a bounded formula. $W(M)$ is also a set, because the quantification over $\mathbb{V}$ can be replaced by quantification over $ran(M)$, the range of $M$; since $M$ is a set, $ran(M)$ is a set by Replacement and the properties of the set-theoretic ordered pair. 

The operator $\Upsilon(M)$ is ``bounded'': for a given $B$, the class of all conclusions $f$ is given by a set, namely $\setb{f \in W(M)}{B \simeq V(f,M)}$; and, the class of all possible premise sets $B$ is given by a set, namely $\setb{V(f,M)}{f \in W(M)}$. Using Regular Extension, it can be shown that if an operator is bounded, the class it inductively defines is actually a set \cite[\S 5]{AR}. \end{proof}

\begin{lem} $M^\infty$ is coherent. \end{lem}
\begin{proof} As with Lemma \ref{lem-c-con}, define $M'$ as the coherent subclass of $M^\infty$. Now suppose there is a set $M \subseteq M'$ and a set $a$ such that $\stuple{M,a} \in \Phi'$. Lemma \ref{lem-c-con} takes care of the $\Xi$ case to show that $a \in M'$. So consider the $\Omega$ case: there is an $f \in W(M)$ and an $h \in \p1$ such that $a \simeq \stuple{\Omega(f),h}$ and $h \simeq \setb{\neumarral{0}}{f \in B(M)^\infty}$. 

By definition of $W(M)$, for all $n \in \omega$, $fn$ is defined and there is a $j$ such that $\stuple{fn,j} \in M$. $M \subseteq M'$, so this $j$ is unique. Now $a \in M^\infty$ since $M \subseteq M^\infty$ and $\stuple{M,a} \in \Phi'$. Suppose there is another $\stuple{\Omega(f),h'} \in M^\infty$. Then there exists a set $N \subseteq M^\infty$ such that $f \in W(N)$, and $h' \simeq \setb{\neumarral{0}}{f \in B(N)^\infty}$. By definition of $W$, for all $n$, $fn$ is defined and there is a $j' \in \p1$ such that $\stuple{fn,j'} \in N$. But these are the same $fn$'s as for $W(M)$, and the corresponding $j$'s are unique, so $j' \simeq j$. Therefore $f \in B(M)^\infty \leftrightarrow f \in B(N)^\infty$, so $h' \simeq h$. 

So again, $h$ being unique, $a \in M'$, and the class $M'$ is $\Phi'$-closed. But $M^\infty$ is the least such class, so all of $M^\infty$ is coherent. \end{proof}

\begin{lem}\label{lem-cr-omega} In $M^\infty$, \eqref{eq-om}, \eqref{eq-oti}, and \eqref{eq-ote} are true. \end{lem}
\begin{proof} Suppose $f$ is a sequence of Meaningful conditions. Then $f \in W(M^\infty)$. By the previous lemma all the corresponding $\stuple{fn,h} \in M^\infty$ are unique, so by Replacement, the following is a set:
$$M_V = \setb{\stuple{fn,j} \in M^\infty}{n \in \omega}$$ 
$B(M_V)^\infty$ is a set by Lemma \ref{lem-cr-bar}, and $M^\infty$ is $\Phi$-closed so $\stuple{\Omega(f),\setb{\neumarral{0}}{f \in B(M_V)^\infty}} \in M^\infty$, so $\Omega(f)$ is Meaningful, as required by \eqref{eq-om}. Furthermore, $f \in B(M_V)^\infty$ expresses the correct Truth conditions for \eqref{eq-oti} and \eqref{eq-ote}, as shown by Lemma \ref{lem-cr-binf}. \end{proof}

\begin{mthm}\label{mthm-cr-ect} Assumptions \eqref{eq-vi}, \eqref{eq-ve}, \eqref{eq-xit}, \eqref{eq-xim}, \eqref{eq-om}, \eqref{eq-oti}, and \eqref{eq-ote}, along with \formalsys{ECT_V}, are true in a realisability interpretation, provided that the seven assumptions hold (without \formalsys{ECT_V}) in the underlying meta-theory. \end{mthm}
\begin{proof} It is the same interpretation as in Meta-theorem \ref{mthm-c-ect}, and the previous proof disposes of \formalsys{ECT_V} and the first four assumptions. Moreover, \eqref{eq-om} and \eqref{eq-oti} are trivially disposed of, being almost-negative, in the same way as \eqref{eq-vi}, \eqref{eq-xit}, and \eqref{eq-xim}. What remains is \eqref{eq-ote}, which in some instances has a non-trivial witness, and it is disposed of in a way similar to \eqref{eq-ve}. Define:
$$\rho = \Lambda e.\Lambda f.(e'f)(\Lambda n.(\rho e)(f^{(n)}))$$
Here $e'$ is a simple syntactic transformation introduced below for clarity of presentation. It will follow that $(\rho e)f$ iterates $e$ by bar recursion along $f$, and as a result, $\rho$ witnesses \eqref{eq-ote} for any $\phi$, provided \eqref{eq-ote} holds in the meta-theory. This is argued in a way which is similar to Lemma \ref{mthm-c-ect}.

Suppose $e$ witnesses the antecedent of \eqref{eq-ote}:
$$e \Vdash \allnat{f} \mathcal{M}(\Omega(f)) \wedge (\mathcal{T}(f\tuple{}) \to \allnat{n} \phi(f^{(n)})) \to \phi(f)$$
which means that \begin{align*}
\allnat{f,a} &(\mathcal{M}(\Omega(f)) \wedge (\mathcal{T}(f\tuple{}) \to \allnat{n} \defd{an} \wedge (an \Vdash \phi(f^{(n)}))) \\
 &\to \defd{(e'f)a} \wedge (e'f)a \Vdash \phi(f)
\end{align*}
The transformation $e \mapsto e'$ causes $e$ to ignore the trivial witnesses for $\mathcal{M}$ and $\mathcal{T}$. Suppose, by way of inductive hypothesis, that $f$ is a sequence of Meaningful conditions, and that:
$$ \mathcal{T}(f\tuple{}) \to \allnat{n} \defd{(\rho e)(f^{(n)})} \wedge (\rho e)(f^{(n)}) \Vdash \phi(f^{(n)}) $$
By \eqref{eq-om}, $\mathcal{M}(\Omega(f))$, so, substituting $\Lambda n.(\rho e)(f^{(n)})$ for $a$ in the assumption on $e$, it follows that:
$$\defd{(e'f)(\Lambda n.(\rho e)(f^{(n)})} \wedge (e'f)(\Lambda n.(\rho e)(f^{(n)}) \Vdash \phi(f)$$
By the definition of $\rho$:
$$\defd{(\rho e)f} \wedge (\rho e)f \Vdash \phi(f)$$
If the induction principle of \eqref{eq-ote} holds in the underlying meta-theory, it follows that:
$$\allnat{f} \mathcal{T}(\Omega(f)) \to \defd{(\rho e)f} \wedge (\rho e)f \Vdash \phi(f)$$
This is the consequent of \eqref{eq-ote}. So, $\rho$ witnesses \eqref{eq-ote}, provided the same schema holds in the underlying meta-theory. \end{proof}

\begin{mthm}\formalsys{CZF+REA} interprets \formalsys{CZF+REA+RDC+ESP} in a way that preserves almost-negative sentences of arithmetic.\end{mthm}
\begin{proof} Lemmas \ref{lem-c-xi} and \ref{lem-c-set} work here as well, and Lemma \ref{lem-cr-omega} uses only arguments available in \formalsys{CZF+REA}. Therefore the seven assumptions can be interpreted into \formalsys{CZF+REA}. The realisability interpretation of Meta-theorem \ref{mthm-cr-ect} preserves almost-negative formulas of arithmetic, and in this interpretation \formalsys{CZF+REA+RDC+ESP} is true. \end{proof}
\begin{rem} Again relative consistency for \formalsys{ESP} \cite[\S 8]{MR}, and for \formalsys{RDC} \cite[\S 5]{RG} was already known, though this combined result is obtained without detours. Here there is no point in considering \formalsys{CZF^-+REA}: \formalsys{REA} implies Subset Collection \cite[\S 10]{AR}. \end{rem}

\section{Extended Church's Thesis and Omniscience}

The assumptions in this paper are unproblematic both classically and constructively, with the exception of \formalsys{ECT}. One of the consequences of even weak constructive forms of Church's Thesis is:
$$\neg \allnat{e,i} ((\somenat{u} \mathbf{T}(e,i,u)) \vee (\neg \somenat{u} \mathbf{T}(e,i,u)))$$
Almost all mathematicians think classically, and form that perspective this assertion, and therefore \formalsys{ECT}, can only be false. But it is provably consistent, and can be understood indirectly in terms of a realisability interpretation in a classical meta-theory, as in section 8. When it is set up this way, not only is \formalsys{ECT} valid, but so is Markov's Principle. No form of this principle was needed in this article, though, so it isn't actually necessary that the meta-theory be classical. This is why things still work when the realisability interpretations are set up in constructive meta-theories, as in sections 7 and 11.

But it isn't actually necessary to have a meta-theory at all. It is possible to understand \formalsys{ECT} literally instead. This does then force literal acceptance of the indeterminacy of the halting problem. Of course this raises some philosophical questions, but at the same time, it disposes of others. A well-known one is the question of non-standard models of first-order arithmetic. It is disposed of by adapting Tennenbaum's theorem to \formalsys{ECT}, from which it follows that there simply aren't any non-standard models \cite[\S 4.1]{CM}. A less well-known one, which perhaps this article will help make better known, is the justification of set theory. The principles of set theory are by no means self-evident, even from a classical point of view. They are taken as axioms. They can be argued for in various ways, such as in terms of their consequences \cite[\S I]{PM}. The fact that they are in turn provable consequences of \formalsys{ECT} plus a simple, precise definition (and perhaps one with no non-standard models) disposes of this issue. Furthermore since the definition is entirely numerical, it also disposes of the ontology of mathematical entities other than the natural numbers. 

Thus the principles of set theory refer to clearly defined objects, they mean what they say, and they are not axioms, but theorems. Taking \formalsys{ECT} at face value allows set theory to be taken at face value. This point of view has some claim to the term {\it realism}. Now, it may be objected that if the natural numbers {\it really exist}, the halting problem simply must be determinate. A good answer to this should involve a proper account of realism in mathematics, and is beyond the scope of this author. Instead, it will be shown that even if determinacy of the halting problem is conceded, not much is lost.

Richman introduced a weaker assumption that has many of the consequences of Church's Thesis \cite{FR}. In summary, often a constructive proof involving Church's Thesis is really only relying on the fact that the computable partial sequences can be computably enumerated. The adjective ``computable'' is actually inessential. It would work just as well with the assumption that the X-able partial sequences can be X-ably enumerated. This is indeed exactly what is being used here to get a numerical interpretation. 

Another way of looking at this is that absolute computability can be replaced with computability relative to some unspecified oracle $\mathcal{O}$. Consider the sets of the form $\setbn{n}{fn}{gn \simeq hn}$, or $\setbn{n}{fn}{(pn \text{ is True})}$, where the sequences are now only required to be computable relative to $\mathcal{O}$. The partial functions which are computable relative to $\mathcal{O}$ can be enumerated computably relative to $\mathcal{O}$, so this can still be formalised entirely numerically. Replace the absolute predicate $\mathbf{T}$ with a relativised predicate $\mathbf{T}^\mathcal{O}$, of which it is only assumed that it is decidable, univocal in its output, and complete under, but not necessarily limited to, the usual computable operations. Define \formalsys{ECT_0^\mathcal{O}} to be the same schema as \formalsys{ECT_0}:
$$(\allnat{n} \phi(n) \to \somenat{m} \psi(n,m)) \to \somenat{e} \allnat{n} \phi(n) \to \defd{en} \wedge \psi(n,en)$$
But now redefine the notations $en$ and $\defd{en}$ to refer to the predicate $\mathbf{T}^\mathcal{O}$ rather than just $\mathbf{T}$. Similarly \formalsys{ECT_V^\mathcal{O}} can be defined, and the reader can check that all the proofs of sections 2-6 go through otherwise unchanged. In the relativised interpretation, some extra sentences might be decidable on account of $\mathcal{O}$, but \formalsys{CZF+RDC+ESP} is still validated. For instance if it is additionally assumed that $\mathcal{O}$ is able to (but again, not necessarily limited to being able to) solve the halting problem, then the result is a consistent extension of \formalsys{CZF+RDC+ESP} with a so-called Omniscience principle:
$$\allnat{x} \phi(x) \vee \neg \phi(x)$$
where $\phi$ is any $\Sigma_1$ sentence of arithmetic, from which of course it follows that the halting problem is determinate.

To see that this really is a consistent theory, interpret it back into plain \formalsys{CZF}. Now, the direct embedding of section 7 will longer work, since $\mathbf{T}^\mathcal{O}$ is not actually decidable for any non-trivial $\mathcal{O}$. On the other hand, the embedding into \formalsys{ID_1} of section 8 was already classical, and stands with just the change in notation mentioned for this $\mathcal{O}$. And if it is further objected that if the natural numbers {\it really exist}, then $\Sigma_2$ sentences must also be determinate, then this embedding can be trivially extended to accommodate:

\begin{mthm}For any fixed $n$, \formalsys{CZF^-} can interpret \formalsys{CZF+RDC+ESP}, plus Omniscience for all $\Sigma_n$ sentences of arithmetic, in a way that preserves the $\Pi_2$ sentences of arithmetic.\end{mthm}
\begin{proof} For $n$ fixed, the truth predicate for $\Sigma_n$ sentences can be defined in classical first-order arithmetic. So an appropriate $\mathbf{T}^\mathcal{O}$ can be non-inductively defined. Therefore $\mathcal{T}$, $\mathcal{M}$ and $\mathcal{V}$ can be defined with a non-nested inductive definitions in \formalsys{ID_1}, allowing \formalsys{CZF+RDC+ESP} plus the required Omniscience principle to be interpreted, as in Meta-theorem \ref{mthm-c-id1}; and again, \formalsys{CZF^-} can indirectly interpret \formalsys{ID_1}. \end{proof}

\begin{cor}\formalsys{CZF^-} proves the same $\Pi_2$ sentences of arithmetic as \formalsys{CZF+RDC+ESP} plus Omniscience for all sentences of first-order arithmetic.\end{cor}
\begin{proof} Omniscience for prenex sentences is sufficient to prove the prenex normal form theorem for all sentences of first-order arithmetic. The conclusion now follows by (constructive) compactness: any given proof of a $\Pi_2$ sentence could only appeal to prenex-Omniscience a finite number of times, so, that proof would work with $\Sigma_n$-Omniscience for some finite $n$.  \end{proof}

It would be nicer to have a direct interpretation for the corollary, but this is not attempted here. There is further to go anyway. Hyperarithmetic sentences can be considered. At some point in the hyperarithmetic hierarchy, the required Omniscience principle necessarily will transcend \formalsys{CZF}. Still, in general, if there is a system with classical logic which can define the truth predicate for a certain class of sentences, and then inductively define $\mathcal{T}$, $\mathcal{M}$ and $\mathcal{V}$ on top of it, then, \formalsys{CZF+RDC+ESP} plus Omniscience for that class of sentences could still be interpreted in that classical system. Then it becomes a matter of interpreting that classical system in a constructive system. This would in particular apply to the hyperarithmetic hierarchy up to $\alpha$. Perhaps even a modest critical point can be expected after which the assumption that $\alpha$ is well-founded would be sufficient to prove the consistency of \formalsys{CZF+RDC+ESP} plus Omniscience for the hyperarithmetic hierarchy up to $\alpha$. 

A proper analysis of this is not attempted here either. The relativisation can be applied even beyond this. For example, it is beyond the ability of anything in the hyperarithmetic hierarchy to decide whether computable relations are well-founded (Kleene's $\mathbf{O}$). Yet there is still a consistent extension of \formalsys{CZF+RDC+ESP} for the corresponding Omniscience principle:

\begin{mthm}\label{mthm-o-ck}\formalsys{ID_2(O)^i} (which is \formalsys{HA} plus an inductive definition for the constructive third number class) can interpret \formalsys{CZF+RDC+ESP}, plus Omniscience for the well-foundedness of computable relations, in a way that preserves $\Pi_2$ sentences of arithmetic.\end{mthm}
\begin{proof} The well-foundedness predicate for computable relations can be defined with a classical non-nested inductive definition. Therefore with a second level of inductive definition, $\mathcal{T}$, $\mathcal{M}$ and $\mathcal{V}$ can be defined over an appropriate $\mathbf{T}^\mathcal{O}$, as in Meta-theorem \ref{mthm-c-id1}. This form of inductive definition is possible in \formalsys{ID_2}, which in turn can be interpreted in \formalsys{ID_2(O)^i} \cite{BFPS}. \end{proof}

It is clear that this doesn't represent the limit of relativisation. It is straightforward to accommodate higher constructive number classes, for instance. 

Now, no matter what, given $\mathbf{T}^\mathcal{O}$, the halting problem relativised to $\mathcal{O}$ can always be expressed, and no oracle can ever solve its own relativised halting problem. So, it always be true that:
$$\neg \allnat{e,i} (\somenat{u} \mathbf{T}^\mathcal{O}(e,i,u)) \vee (\neg \somenat{u} \mathbf{T}^\mathcal{O}(e,i,u))$$
In other words, {\it some} relativised halting problem is still indeterminate. But only assumptions on what $\mathcal{O}$ can decide ever need to be made, never assumptions on what $\mathcal{O}$ is limited to. Therefore if something about the natural numbers ought to be determinate, it can be assumed so, and it seems that this relativised indeterminacy is of no consequence.

At least, there is no consequence as long as only the natural numbers are considered. The story is different for the real numbers. There is a consequence of the above, which therefore cannot be avoided by any relativisation:
$$\neg \allin{x}{{}^{\mathbb{N}}\mathbb{N}} (\somenat{n} x(n) = 0) \vee (\neg \somenat{n} x(n) = 0)$$
This is the negation of the so-called Limited Principle of Omniscience (\formalsys{LPO}). Even weaker Omniscience principles of this kind are considered and refuted by Richman \cite{FR2} using weaker forms of Church's thesis, and the refutations may still work with the relativised form. The difference in kind here is that \formalsys{LPO} and its weakenings no longer refer to a well-defined class of sentences involving natural numbers, rather, they demand a certain type of determinacy for arbitrary sequences of numbers. The consequences of \formalsys{LPO} and its kind are felt in analysis rather than number theory \cite{HI}. 

Therefore I believe that if there is a philosophical issue here, it has to come from an ontology of real numbers, which might demand \formalsys{LPO}, or of sets in general, which might require $\p1 \simeq \neumarral{2}$ or the existence of classically uncountable (that is, non-subcountable) sets. There should be no objection solely on the basis of an ontology of the natural numbers, which the good lord created.

\section*{Acknowledgments}
I would like to thank Prof. Peter Aczel for helpful exchanges on this topic, including constructive (and predicative) criticisms of various draft versions of this article.

\end{document}